\documentclass[11pt]{article}

\usepackage{mathrsfs}

\usepackage[T1]{fontenc}
\usepackage[utf8]{inputenc}
\usepackage{lmodern}
\usepackage{amssymb,amsmath,amsthm}
\usepackage{bbm}
\usepackage{graphicx}
\usepackage{float}
\usepackage{xcolor}
\usepackage[colorlinks,linkcolor=black!70!black,citecolor=black]{hyperref} 
\usepackage[a4paper,margin=1.5cm]{geometry}
\usepackage{tikz}
\usepackage{tikz-cd}

\usetikzlibrary{decorations.pathreplacing, patterns,shapes,snakes}

\usepackage{pgfplots}

\long\def\metanote#1#2{{\color{#1}\
\ifmmode\hbox\fi{\sffamily\mdseries\upshape [#2]}\ }}
\long\def\OT#1{\metanote{blue!70!black}{{\tiny OT} #1}}

\newcommand{\ra}{\rightarrow}

\newcommand{\be}{\begin{equation}}
\newcommand{\ee}{\end{equation}}
\newcommand{\bi}{\begin{itemize}}
\newcommand{\ei}{\end{itemize}}

\newcommand{\commentout}[1]{}

\newcommand{\CDF}{\text{CDF}}
\newcommand{\TV}{\text{TV}}

\newcommand{\Wah}{\text{W}}

\newcommand{\Law}{{\mathcal{L}}}

\newcommand{\Leb}{{\text{Leb}}}

\newcommand{\calD}{{\mathcal{D}}}

\newcommand{\calO}{{\mathcal{O}}}

\newcommand{\calP}{{\mathcal{P}}}

\newcommand{\calS}{{\mathcal{S}}}

\newcommand{\Nm}{{\mathbb{N}}}

\newcommand{\Rm}{{\mathbb R}}

\newcommand{\Pm}{{\mathbb P}}

\newcommand{\expE}{{\mathbb E}}
\newcommand{\Ind}{{\mathbbm{1}}}

\newtheorem{theo}{Theorem}[section]
\newtheorem{lem}[theo]{Lemma}
\newtheorem{defin}[theo]{Definition}

\newtheorem{prop}[theo]{Proposition}

\newtheorem{rmk}[theo]{Remark}

\newtheorem{conj}[theo]{Conjecture}

\commentout{

\newcommand{\calD}{{\mathcal{D}}}

\newcommand{\calO}{{\mathcal{O}}}

\newcommand{\calP}{{\mathcal{P}}}
\newcommand{\calS}{{\mathcal{S}}}

\newcommand{\Nm}{{\mathbb{N}}}

\newtheorem*{rep@theorem}{\rep@title}
\newcommand{\newreptheorem}[2]{%
\newenvironment{rep#1}[1]{%
 \def\rep@title{#2 \ref{##1}}%
 \begin{rep@theorem}}%
 {\end{rep@theorem}}}
\makeatother
\newreptheorem{theorem}{Theorem}
\newreptheorem{lemma}{Lemma}
\newreptheorem{defin}{Definition}
\newreptheorem{cond}{Condition}
\newreptheorem{prop}{Proposition}
\newreptheorem{lem}{Lemma}
\newreptheorem{cor}{Corollary}
\newreptheorem{rmk}{Remark}
\newreptheorem{exer}{Exercise}
\newreptheorem{conj}{Conjecture}
\newreptheorem{assum}{Assumption}
\newreptheorem{equation}{Equation}
\newreptheorem{problem}{Problem}

\long\def\metanote#1#2{{\color{#1}\
\ifmmode\hbox\fi{\sffamily\mdseries\upshape [#2]}\ }}
\long\def\OT#1{\metanote{blue!70!black}{{\tiny OT} #1}}
\long\def\JB#1{\metanote{red!70!black}{{\tiny JB} #1}}
}

\newtheorem*{rep@theorem}{\rep@title}
\newcommand{\newreptheorem}[2]{%
\newenvironment{rep#1}[1]{%
 \def\rep@title{#2 \ref{##1}}%
 \begin{rep@theorem}}%
 {\end{rep@theorem}}}
\makeatother
\newreptheorem{theorem}{Theorem}
\newreptheorem{lemma}{Lemma}
\newreptheorem{defin}{Definition}
\newreptheorem{cond}{Condition}
\newreptheorem{prop}{Proposition}
\newreptheorem{lem}{Lemma}
\newreptheorem{cor}{Corollary}
\newreptheorem{rmk}{Remark}
\newreptheorem{exer}{Exercise}
\newreptheorem{conj}{Conjecture}
\newreptheorem{assum}{Assumption}
\newreptheorem{equation}{Equation}

\long\def\metanote#1#2{{\color{#1}\
\ifmmode\hbox\fi{\sffamily\mdseries\upshape [#2]}\ }}
\long\def\OT#1{\metanote{green!55!black}{{\tiny OT} #1}}
\long\def\JB#1{\metanote{red!70!black}{{\tiny JB} #1}}

\begin{document}
\setcounter{page}{1}

\title{Selection principle for the $N$-BBM}
\author{Julien Berestycki\footnote{Department of Statistics and Magdalen College, University of Oxford, berestyc@stats.ox.ac.uk} and Oliver Tough\footnote{Department of Mathematics, University of Bath, oliverktough@gmail.com}}
\date{\today}

\maketitle
\begin{abstract}
    The $N$-branching Brownian motion with selection ($N$-BBM) is a particle system consisting of $N$ independent particles that diffuse as Brownian motions in $\mathbb{R}$, branch at rate one, and whose size is kept constant by removing the leftmost particle at each branching event. We establish the following selection principle: as $N\ra\infty$ the stationary empirical measure of the $N$-particle system converges to the minimal travelling wave of the associated free boundary PDE. This resolves an open question going back at least to \cite[p.19]{Maillard2012} and \cite{GroismanJonckheer}, and follows a recent related result by the second author establishing a similar selection principle for the so-called Fleming-Viot particle system \cite{Tough23}.
\end{abstract}

\section{Introduction and main results}\label{section:introduction}

The $N$-branching Brownian motion with selection (or $N$-BBM for short), is one of the simplest examples of a \emph{branching particle system with selection}. It consists of $N$ particles that move as standard Brownian motions and independently branch at rate 1. At each branching event, the leftmost particle is removed from the system and its label is given to the newborn particle, thus keeping the population size constantly equal to $N$. Equivalently, at rate $N-1$ the minimal particle jumps to the location of one of the remaining $N-1$ particles chosen uniformly at random. We denote the $N$-BBM by $\vec{X}^N_t=(X^1_t,\ldots,X^{N}_t)$, for $N\geq 1$. The driving Brownian motions are denoted by $W^i_t$ ($1\leq i\leq N$), so that $dX^i_t=dW^i_t$ in between jump times. 

We define 
\[
L^N_t:=\min_{1\leq i\leq N}X^i_t,\quad Y^i_t:=X^i_t-L^N_t,\quad 1\leq i\leq N,\quad \text{and}\quad \vec{Y}^N_t:=(Y^1_t,\ldots,Y^N_t).
\]
Then $\vec{Y}^N_t$ is the $N$-BBM, recentred so that the leftmost particle is at $0$. It is a strong Markov process on the state space
\begin{equation}\label{eq:GammaN defin}
\Gamma_N:=\{\vec{y}=(y_1,\ldots,y_N)\in \Rm_{\geq 0}^N:y_i=0\quad\text{for some}\quad 1\leq i\leq N\}.
\end{equation}

We write $\Pm_{\vec{x}^N}$ for the probability measure under which $\vec{X}^N_t$ is an $N$-BBM with $\vec{X}^N_0=\vec{x}^N$ almost surely. 


Our first result says that for fixed $N$, the centred process $\vec{Y}^N_t$ converges to its unique stationary distribution $\psi^N$.
\begin{theo}\label{theo:convergence to stationary distribution for fixed N}
For any fixed $1\leq N<\infty$, $\vec{Y}^N_t$ satisfies Doeblin's condition: there exists $c_N>0$ and $\alpha_N\in \calP(\Gamma_N)$ such that
\begin{equation}
\Law_{\vec{y}^N}(Y^N_1)(\cdot)\geq c_N\alpha_N(\cdot)\quad\text{for all}\quad \vec{y}^N\in \Gamma_N.
\end{equation}
It follows that there exists a unique stationary distribution for $(\vec{Y}^N_t)_{t\geq 0}$, which we denote by $\psi^N$, and that 
\begin{equation}\label{eq:convergence in total vairation for fixed N}
    \lvert\lvert \Law_{\vec{y}^N}(Y^N_t)(\cdot)-\psi^N(\cdot)\rvert\rvert\leq (1-c_N)^{\lfloor t\rfloor}\quad \text{for all}\quad t\geq 0,\; \vec{y}\in \Gamma_N.
\end{equation}
\end{theo}

\begin{proof}[Proof of Theorem \ref{theo:convergence to stationary distribution for fixed N}]
There is a probability, uniformly bounded away from $0$ over all initial conditions, that in time $\frac{1}{2}$ there are $N$ killing events, in each of which the minimal particle jumps onto the maximal particle at the time, and that $\sup_{1\leq i\leq N}\sup_{0\leq s,t\leq 1}\lvert W^i_t-W^i_s\rvert\leq 1$.

On this event one can check that there is a compact set $K\subseteq \Gamma_N$ such that $\vec{Y}^N_{\frac{1}{2}}\in K$, whatever the initial condition. The claim then follows from the parabolic Harnack inequality.
\end{proof}

De Masi et al. established \cite[Theorem 1]{DeMasi2019} that the $N$-BBM has the following hydrodynamic limit. Consider a sequence of $N$-BBMs such that the initial positions of the particles are independent and identically distributed with some common distribution $u_0$, (where $u_0$ also satisfies certain assumptions). Then as $N\to \infty$ the empirical measure of the particle system converges to the solution of the following free boundary PDE,
\begin{equation}\label{eq:hydrodynamic limit}
\begin{cases}
\partial_tu=\frac{1}{2}\Delta u+u,\quad x>L_t,\\
u(L_t,t)=0,\\
\int_{L_t}^{\infty}u(t,y)dy=1\quad\text{for all}\quad t\geq 0\\
u(t,\cdot) \to u_0(\cdot) \text{ weakly as } t\to 0.
\end{cases}
\end{equation}
Here, the unknowns are both the function $u: \Rm\times \Rm_+ \mapsto [0,1]$ and the boundary $L: \Rm_+\mapsto \Rm.$
By $u(t,\cdot)\to u_0(\cdot)$ we mean that the measure $u_t(\cdot)d\Leb(\cdot)$ converges to the measure $u_0(\cdot)$ in the sense of weak convergence of measures ($\Leb$ being Lebesgue measure on $\Rm$). Global existence and uniqueness of the solutions of \eqref{eq:hydrodynamic limit} for any initial distribution was established in \cite[Theorem 1.1 and Corollary 2.1]{Berestycki2018}. 


A similar hydrodynamic limit for another branching-selection particle system was established earlier by Durrett and Reminik in \cite{Durrett2011}. The related boundary value problem where $L$ is given was also studied using probabilistic tools in \cite{BBHR17} and analytically in \cite{Henderson17}.

The free boundary problem \eqref{eq:hydrodynamic limit} is in the same universality class as the celebrated Fisher-KPP equation since it has the necessary three main ingredients: diffusion, growth and saturation. For instance it is easy to see that, like the Fisher-KPP equation, \eqref{eq:hydrodynamic limit} has a family of travelling wave solutions,
\[
u_c(x,t)=\pi_c(x-ct,t)\quad\text{for}\quad c\geq \sqrt{2}.
\]
The minimal travelling-wave solution - the travelling-wave solution with minimal wave speed - is $\pi_{\min}=\pi_{\sqrt{2}}$. It is given by
\begin{equation}\label{eq:minimal travelling wave}
\pi_{\min}(x)=2xe^{-\sqrt{2}x}, \quad x>0,
\end{equation}
and its wave speed is $c_{\min}=\sqrt{2}$. It is believed that the similarity of \eqref{eq:hydrodynamic limit} to the Fisher-KPP equation holds in much greater generality, see for instance Conjecture \ref{conj:dmain of attraction free boundary PDE}.

The purpose of the present article is to show that the stationary distribution $\psi^N$ converges to $\pi_{\min}$ in the appropriate sense. This is called a \emph{strong selection principle} since the particle system selects the minimal-speed travelling wave. 

To that end, denote by $\Theta^N:\Rm^N\ra \calP(\Rm)$ the map
\[
\Theta^N:(x_1,\ldots,x_{N})\mapsto \frac{1}{N}\sum_{i=1}^N\delta_{x_i}\in \calP(\Rm).
\]
Our main result is the following theorem.
\begin{theo}[Selection principle for the $N$-BBM]\label{theo:main theorem 0}
Let $\vec{Y}^N \sim \psi^N$ for $N<\infty$. Then 
\[
\Theta^N(\Vec{Y}^N) \to \pi_{\min} 
\]
weakly in probability as $N\to \infty$. 
\end{theo}

This result as actually a direct consequence of a slightly stronger result. 
We recall the definition of the Wasserstein metrics.
\begin{defin}\label{defin:Wasserstein}
Throughout, given a complete metric space $(E,d)$, we define $\calP(E)$ and $\calP_1(E)$ to be the space of Borel probability measures on $E$ and the space of Borel probability measures with finite first moment respectively. In the latter case this means that $\calP_1(E):=\{\mu\in \calP(E):\int_Ed(x_{\ast},x)\mu(dx)<\infty\}$, where $x_{\ast}\in E$ is some fixed distinguished point (note that this definition does not depend upon the choice of $x_{\ast}$). We then define on $\calP(E)$ and $\calP_1(E)$ respectively the Wasserstein metrics $\Wah$ and $\Wah_1$ by
\begin{align*}
\Wah(\mu_1,\mu_2)&:=\inf_{\substack{\nu\in \calP(E\times E)\text{ s.t.}\\((x,y)\mapsto x)_{\#}\nu=\mu_1,\\((x,y)\mapsto y)_{\#}\nu=\mu_2}}\int_{E\times E}[d(x,y)\wedge 1]d\nu ((x,y)),\\
\Wah_1(\mu_1,\mu_2)&:=\inf_{\substack{\nu\in \calP(E\times E)\text{ s.t.}\\((x,y)\mapsto x)_{\#}\nu=\mu_1,\\((x,y)\mapsto y)_{\#}\nu=\mu_2}}\int_{E\times E}d(x,y)d\nu ((x,y)).
\end{align*}
Then $\Wah$ metrises on $\calP(E)$ the topology of weak convergence of measures, whilst $\Wah_1$ metrises weak convergence in $\calP_1(E)$, meaning both weak convergence of measures and convergence of the first moment (see \cite[Definition 6.8 and Theorem 6.9]{Villani2009}).
\end{defin}

Then Theorem \ref{theo:main theorem 0} is an immediate consequence of the following, which we shall prove in Section \ref{section:proof}.
\begin{theo}\label{theo:main theorem}
We have the convergence 
\[
\expE_{ \vec{Y}^N\sim\psi^N}[\Wah_1(\Theta^N(\vec{Y}^N),\pi_{\min})]\ra 0
\]
as $N\ra\infty$. 
\end{theo}

We summarise the situation in the following diagram.
\begin{figure}[H]
\begin{center}
\begin{tikzcd}[column sep=6cm, row sep=3cm]
\substack{\text{$N$-BBM}} \arrow[r, "\text{\cite[Theorem 1]{DeMasi2019}}\quad (N\ra\infty)"] \arrow[d, "\substack{\text{Theorem \ref{theo:convergence to stationary distribution for fixed N}}\\(t\ra\infty)}"]
& \substack{\text{Solutions of \eqref{eq:hydrodynamic limit}/\eqref{eq:hydrodynamic limit CDF}}} \arrow[d, "\substack{\text{Conjecture \ref{conj:dmain of attraction free boundary PDE}}\\(t\ra \infty)}"] \\
\substack{\text{Stationary Distribution}\\ \text{for the $N$-BBM, $\psi^N$}} \arrow[r, "\text{Theorem \ref{theo:main theorem 0}/\ref{theo:main theorem}}\quad  (N\ra\infty)"]
& \text{Minimal travelling wave, $\pi_{\min}$}
\end{tikzcd}
\end{center}
\end{figure}

It is notable that we prove Theorem \ref{theo:main theorem} without knowing that the PDE \eqref{eq:hydrodynamic limit} converges in any sense to the minimal travelling wave, except for the particular case $u_0=\delta_0$ (or equivalently $U_0(x)=\Ind(x<0)$ for the integrated version of the equation). This is in contrast to the proof of the selection principle for the Fleming-Viot particle system with drift $-1$ \cite{Tough23}, and to other proofs in the literature (see, for example, the proof of the selection principle for the Brownian bees particle system \cite{Berestycki2022}), for which understanding the long-time behaviour of  the limiting PDE is an essential element of the proof. This is made possible by the argument employed in Subsection \ref{subsection:conclusion of proof}.

As indicated on the right-hand side arrow of the above diagram, it remains an open problem to characterise initial conditions for which the solutions of \eqref{eq:hydrodynamic limit} converge to the minimal travelling wave $\pi_{\min}$. 

Writing $U(x,t):=\int_x^{\infty}u(y,t)dy$ and $U_0(x):=\int_x^{\infty}u_0(dy)$, \eqref{eq:hydrodynamic limit} is equivalent to the following integrated version of the problem
\begin{equation}\label{eq:hydrodynamic limit CDF}
\begin{cases}
\partial_tU=\frac{1}{2}\Delta U+U,\quad x>L_t,\quad t>0,\\
U(x,t)= 1,\quad x\leq L_t,\quad t>0,\\
\partial_xU(L_t,t)= 0,\quad t>0\\
U(t,x)\ra U_0(x)\text{ pointwise at all continuity points of $U_0$ as $t\ra 0$.}
\end{cases}
\end{equation}

Basic properties of solutions of \eqref{eq:hydrodynamic limit CDF} are given by \cite{Berestycki2018}. We also have the following.
\begin{rmk}\label{rmk:U strictly decreasing at pve times}
Any solution $(u,L)$ of \eqref{eq:hydrodynamic limit} is strictly positive on $\{(x,t):x>L_t,t>0\}$, by the parabolic Harnack inequality, hence any solution $(U,L)$ of \eqref{eq:hydrodynamic limit CDF} is strictly decreasing on $\{x:x>L_t\}$, for any $t>0$.
\end{rmk}

The corresponding travelling waves are then given by $\Pi_c(x):=\int_x^{\infty}\pi_c(y)dy$ for $c\geq c_{\min}=\sqrt{2}$, so that the minimal travelling wave for \eqref{eq:hydrodynamic limit CDF} is given by
\begin{equation}
\Pi_{\min}(x):=\int_x^{\infty}\pi_{\min}(y)dy.
\end{equation}
By analogy with both the Fisher-KPP equation (see Subsection \ref{subsection:Background and related results}) and the law of Brownian motion in $\Rm_{>0}$ with constant negative drift conditioned not to hit $0$, we conjecture the following. 
\begin{conj}\label{conj:dmain of attraction free boundary PDE}
The following are equivalent:
\begin{enumerate}
    \item $\limsup_{x\ra\infty}\frac{1}{x}\ln U_0(x)\leq -\sqrt{2}$;
    \item $\limsup_{t\ra\infty}\frac{L_t}{t}\leq \sqrt{2}$;
    \item $\lim_{t\ra\infty}\frac{L_t}{t}= \sqrt{2}$;
    \item $U(L_t+x,t)\ra \Pi_{\min}(x)$ uniformly in $x$ as $t\ra\infty$.\label{enum:convergence to the minimal travelling wave}
\end{enumerate}
\end{conj}

One can think about the free boundary PDE \eqref{eq:hydrodynamic limit} in the following way. Let $(u,L)$ be its solution (with some initial condition $u_0$), and now consider a Brownian motion $B_t$ started from a random initial position with distribution $u_0$ and killed upon hitting the boundary $L_t$. Then $u_t$ is the density of $B_t$ conditioned to be alive at time $t$. Since traveling wave solutions to \eqref{eq:hydrodynamic limit} have linear free boundaries $L_t=ct, c\ge \sqrt 2$, the  travelling wave solution with speed $c$ corresponds to the quasi-stationary distribution (QSD) for Brownian motion with drift $-c$ killed at 0 for which the absorption time is exponential with parameter one (which after rescaling is also the QSD for Brownian motion with drift $-1$ for which the absorption time is exponential with rate $c^{-2}$). See \cite[Section 4]{GroismanJonckheer} for a deeper discussion of this. Interestingly, it was shown in \cite[Theorem 1.3]{Martinez1998} that if one starts a Brownian motion with drift $-1$ killed upon hitting $0$, $(X_t:0\leq t<\tau_0)$, from some distribution $u_0$ satisfying the  tail condition \eqref{eq:tail for FKPP to converge to travelling wave} appearing for the FKPP, then $\Law_{u_0}(X_t\lvert\tau_0>t)$ converges to the minimal QSD as $t\ra \infty$. This condition was shown to be necessary in \cite[Theorem 1.1]{Tough23}, and is analogous to our Conjecture \ref{conj:dmain of attraction free boundary PDE}.



\subsection{Proof strategy and organisation of the paper}

Let us describe informally how the proof of Theorem \ref{theo:main theorem} works. 

The first order of business is to prove the tightness of the $\psi^N$s. The main tool for accomplishing this will be to apply Birkhoff's theorem. It is easy to establish that the particle system has an almost-sure asymptotic velocity $v_N$ (see \eqref{eq:almost sure limit of lower boundary}) for any initial condition, which is at most $\sqrt{2}$. Birkhoff's theorem tells us that almost-sure long-time limits are equal to the corresponding expectation under the stationary distribution - \textit{so if we know a given long-time almost-sure limit, then we know the corresponding expectation under the stationary distribution}. Therefore if the starting configuration is drawn from $\psi^N$, by Birkhoff's theorem, the expected velocity of the barycentre of the cloud of particles is $v_N\le \sqrt 2.$ Then by a martingale argument, we show that this value is also the expectation of the barycentre of a cloud of particles distributed according to $\psi^N$. This implies tightness. This is accomplished in Propositions \ref{prop:N-particle velocity bounded by sqrt 2} and \ref{prop:chiN tight}.

The next step is to characterize sub-sequential limits. To do this we let $\zeta$ be an arbitrary sub-sequential limit. {\it A priori} $\zeta$ is a random probability measure, and we want to show that it is almost surely equal to $\pi_{\min}$. For any given initial condition $\mu$, the free boundary PDE gives us a boundary $L_t(\mu)$ and defines a flow after appropriate recentring. We then use the hydrodynamic limit theorem \cite[Theorem 1]{DeMasi2019} to say that $\zeta$ must be (see Proposition \ref{prop:characterisation of subsequential limits}):
\begin{enumerate}
    \item stationary under the flow given by the free boundary PDE;
    \item such that $\expE_{\mu\sim \zeta}[L_t(\mu)-L_0(\mu)]\leq \sqrt{2}t$, where randomness here is given by the randomness of the initial condition.
\end{enumerate}

We then adapt classical PDE arguments (the so-called stretching Lemma), combined with a stationarity argument, to obtain that $L_t(\mu)-L_0(\mu)\geq \sqrt{2}t$ for $\zeta$-almost every $\mu$. We conclude that $L_t(\mu)\equiv \sqrt{2}t$ almost surely, and from this the conclusion follows (see Propososition \ref{prop:possible invariant measure unique} and Lemma \ref{lem:free boundary a.s. sqrt 2 t}).

The overall structure and some key ideas employed here are similar to those employed  by the second author to prove an analogous result for the Fleming-Viot particle system with drift $-1$, see \cite{Tough23}. However, different techniques must be developed to make the general strategy work in this new context. More precisely,  the most novel aspect of the present proof is the use of PDE arguments to characterise sub-sequential limits. These are needed to replace QSD results from \cite{Martinez1998} which are key to characterising sub-sequential limits in \cite{Tough23} but which are not available here. Moreover the proof of compactness employed here is quite different.

\subsection{Background and related results}\label{subsection:Background and related results}

Theorem \ref{theo:main theorem 0} provides a so-called \textit{selection principle} for the $N$-BBM. This resolves an open question going back at least to Maillard (\cite[p.19]{Maillard2012}, \cite[p.1066]{Maillard2016}) and Groisman and Jonckeere \cite[p.251]{GroismanJonckheer}, having also been conjectured by N. Berestycki and Zhao \cite[p.659]{Berestycki2018b}, and De Masi, Ferrari, Presutti and Soprano-Loto \cite[p.548]{DeMasi2019}. It follows a recent related result of the second author establishing a similar selection principle for the Fleming-Viot particle system with drift $-1$ \cite{Tough23}. This is a different particle system arising in a different context, but which nevertheless bears clear similarities to the $N$-BBM and for which there was an analogous selection problem. Groisman and Jonckeere have produced an excellent survey \cite{GroismanJonckheer} on the selection problems for these two particle systems, and the relationship between them.


The selection problem arose in the context of front propagation. The Fisher-KPP equation,
\begin{equation}\label{eq:FKPP equation}
\frac{\partial u}{\partial t}=\frac{1}{2}\frac{\partial^2 u}{\partial x^2}+u(1-u),
\end{equation}
was introduced independently in 1937 by Fisher \cite{Fisher1937} and Kolmogorov, Petrovskii and Piskunov \cite{Kolmogorov1937} as a model for the spatial spread of an advantageous allele. It was independently shown by both to have an infinite family of travelling wave solutions - solutions of the form $u_c(t,x)=w_c(x-ct)$ - for all wave speeds $c\geq c_{\min}=\sqrt{2}$, but not for any wave speed less than $\sqrt{2}$. Kolmogorov, Petrovskii and Piskunov \cite{Kolmogorov1937} established that, starting a solution $u$ of \eqref{eq:FKPP equation} from a Heaveside step function, there exists $\sigma(t)=\sqrt{2}t+o(t)$ such that $u(x+\sigma(t),t)$ converges to $w_{v_{\min}}(x)$. Bramson \cite{Bramson1978,Bramson1983} refined the speed to $\sigma(t)=\sqrt{2}t-\frac{3}{2\sqrt{2}}\log t+\mathcal{O}(1)$, and showed that the domain of attraction is given by initial conditions $u_0$ such that $\liminf_{x\ra -\infty}\int_{x-H}^xu_0(y)dy>0$ for some $H<\infty$, and 
\begin{equation}\label{eq:tail for FKPP to converge to travelling wave}
\limsup_{x\ra\infty}\frac{1}{x}\log \Big[\int_x^{\infty}u_0(y)dy\Big]\leq -\sqrt{2}.
\end{equation}
Therefore we have convergence to the travelling wave with minimal wave speed when the initial condition has sufficiently light tails. This is a \textit{macroscopic selection principle}. 

This partly motivates Conjecture \ref{conj:dmain of attraction free boundary PDE}, since the free boundary PDE \eqref{eq:hydrodynamic limit} is in the same universality class as the Fisher-KPP equation. In the physics literature, the first author, Brunet and Derrida in \cite{Berestycki2018a} have non-rigorously derived the asymptotics of the free boundary $L_t$, which match those of the Fisher-KPP equation. It is an open problem to make this rigorous.

The aforedescribed PDEs are \textit{deterministic}, with the minimal travelling wave being ``selected'' by virtue of the initial condition having a sufficiently light tail. We contrast this with a \textit{microscopic selection principle}, in which the travelling wave is ``selected'' by virtue of the introduction of a microscopic amount of random noise. This noise can be incorporated by considering either a stochastic PDE or an interacting particle system. 


The first \textit{weak microscopic selection principle} is due to Bramson et al. \cite{Bramson1986} in 1986. They considered a system (very different to the one we consider here) parametrised by a parameter $\gamma<\infty$ (large $\gamma$ representing small noise) which has a hydrodynamic limit given by a reaction-diffusion equation as $\gamma\ra\infty$. They showed that for all $\gamma<\infty$ this system, seen from its rightmost particle, has a unique invariant distribution. Then they showed that the velocity of this stationary distribution, appropriately rescaled, converges to the minimal wave speed of the corresponding reaction-diffusion equation. This is a \textit{weak selection principle} since they established convergence of the wave speed but not of the profile of the stationary distribution. 

Starting in the nineties, the work of Brunet and Derrida et al. \cite{Brunet1997,Brunet1999,Brunet2001,Brunet2006,Brunet2007} engendered a huge growth in the study of the effect of noise on front propagation. The $N$-BBM considered in this paper belongs to a class of branching-selection particle systems introduced by Brunet and Derrida \cite{Brunet1997,Brunet1999}. It has an asymptotic velocity $v_N$ (see \eqref{eq:almost sure limit of lower boundary} for a formal definition). Brunet and Derrida conjectured that
\begin{equation}\label{eq:BerardGouerevelocityNBBM}
v_N=\sqrt{2}-\frac{\pi^2}{\sqrt{2} \log^2N}+o((\log N)^{-2}).
\end{equation}
This conjecture includes not only the statement that $v_N\ra \sqrt{2}$ (a weak selection principle), but that the rate of convergence is given by a (surprisingly large) $-\frac{\pi^2}{\sqrt{2} \log^2N}$ correction term. This conjecture was proven for the $N$-branching random walk by Bérard and Gouéré in \cite{Berard2010}. Brunet and Derrida made a similar prediction for stochastic PDEs, which was proven in great generality by Mueller, Mytnik and Quastel \cite{Mueller2011}.

In contrast a \textit{strong selection principle}, in which one establishes that the profile of the $N$-particle stationary distribution converges to that of the minimal travelling wave as in Theorem \ref{theo:main theorem}, had not been established for any particle system in the travelling wave setting until the present article, and the recent result of the second author establishing a strong selection principle for the Fleming-Viot particle system with drift $-1$ \cite{Tough23}. 

We finally mention a recent similar result for the Brownian bees particle system by the first author, Brunet, Nolen and Penington \cite{Berestycki2022}. This is a variant of the $N$-BBM whereby, instead of killing the leftmost particle at each selection step, one instead kills the particle furthest away from $0$. Under this dynamic, the particles tend to stay near the origin, allowing for a proof strategy which is very different from that of the present article. The analogue of travelling waves in this context is the principal Dirichlet eigenfunction on a ball of uniquely determined radius. No selection is involved since this eigenfunction is unique - there's no analogue of non-minimal travelling waves.

\section{Proof of Theorem \ref{theo:main theorem}}\label{section:proof}

We write $\tau_n^i$ and $\tau_n$ for the $n^{\text{th}}$ killing time of particle $X^i$ (respectively of all particles in the $N$-BBM) for $n\geq 1$ and $i\in \{1,\ldots,N\}$, with $\tau_0,\tau^i_0:=0$. We write $N^i_t:=\sup\{n\geq 0:\tau^i_n\leq t\}$ for $i\in\{1,\ldots,N\}$ and $N_t:=\sup\{n\geq 0:\tau_n\leq t\}$, counting the number of deaths of particle $i$ and of all particles up to time $t$ respectively. 

We recall from \eqref{eq:GammaN defin} that $\Gamma_N:=\{\vec{y}=(y_1,\ldots,y_N)\in \Rm_{\geq 0}^N:y_i=0\quad\text{for some}\quad 1\leq i\leq N\}$. We then define
\begin{equation}\label{eq:empirical mean}
b(\vec{y}):=\frac{1}{N}\sum_{i=1}^Ny^i,\quad \vec{y}=(y^1,\ldots,y^N)\in \Gamma_N,
\end{equation}
giving the empirical mean of the particle system when it is recentred so that the minimal particle is at $0$. 

\begin{rmk}\label{rmk:unif integrability from the left of leftmost}
Observe that writing  $(L_1-L_0)_-:=\lvert (L_1-L_0)\wedge 0\rvert$, an easy coupling argument shows that
\[
\expE_{\vec{x}^N} \left[(L_1-L_0)_- \right] \le \expE_{(0,\ldots , 0)} \left[  (L_1-L_0)_- \right]<\infty.
\]
\end{rmk}

\subsection{The asymptotic velocity of the leftmost particle is at most $\sqrt{2}$}
The following proposition ensures that the asymptotic velocity of the $N$-BBM exists and is at most $\sqrt{2}$. It also ensures that at stationarity, the distance between the rightmost and leftmost particle is an integrable random variable.
\begin{prop}\label{prop:N-particle velocity bounded by sqrt 2}
For $N$ fixed, 
\begin{equation}\label{eq:almost sure limit of lower boundary}
\frac{1}{t}L^N_t\ra v_N  \leq \sqrt{2}=c_{\min}  \quad\text{almost surely as}\quad t\ra\infty.
\end{equation}
The constant $v_N \in (0,\sqrt 2)$ is independent of the initial condition. Furthermore,
\begin{equation}\label{eq:vN equal to expectation}
    v_N=\expE_{\vec{Y}^N_0\sim \psi^N}[L^N_1-L^N_0].
\end{equation}
Finally, we have that 
\begin{equation}\label{eq:rightmost minus leftmost integrable}
    \expE_{\vec{Y}^N\sim \psi^N}[\max_{1\leq i\leq N}Y^{i}]<\infty\quad\text{for all}\quad N<\infty.
\end{equation}
\end{prop}

\begin{proof}[Proof of Proposition \ref{prop:N-particle velocity bounded by sqrt 2}]
We firstly establish that
\begin{equation}\label{eq:lim sup of velocity at most sqrt 2}
\limsup_{t\ra\infty}\Big(\frac{L_t^N-L_0^N}{t}\Big)\leq \sqrt{2}\quad\text{almost surely, for any initial condition.}
\end{equation}
B\'erard and Gou\'er\'e have provided a proof of this fact for a similar branching-selection particle system, with the asymptotic velocity having a $(\log N)^{-2}$ correction \cite{Berard2010} (see \eqref{eq:BerardGouerevelocityNBBM}). Whilst one could readily extend their proof to the $N$-BBM, we are not aware of this having been done. We therefore provide here a proof of \eqref{eq:lim sup of velocity at most sqrt 2}, which is made much simpler by the fact that we are not after the $(\log N)^{-2}$ correction factor. Moreover, we find the following argument to be transparent and clearly robust to changes in the selection mechanism that one might consider.

We couple the $N$-BBM with a branching Brownian motion as follows. At time $0$ there are $N$ red particles, which evolve as an $N$-BBM for all time. Initially,  no blue particle is present. At rate $N$ the leftmost red particle jumps to the location of one of the red particles chosen independently and uniformly at random (including itself). When this happens, we simultaneously add a blue particle at the position of the leftmost red particle just before the jump. Each blue particle henceforth evolves independently as a branching Brownian motion. We see that the red particles form an $N$-BBM while the set of all particles (blue and red) form a branching Brownian motion. 

It follows from this coupling that the almost sure limit $\limsup_{t\ra\infty}\frac{L^N_t-L^N_0}{t}$ is at most the asymptotic velocity of the rightmost particle of a branching Brownian motion starting from $N$ particles. As is well-known, the latter is equal to $\sqrt 2$.


Having fixed $N<\infty$, Remark \ref{rmk:unif integrability from the left of leftmost} provides for the uniform integrability of the negative part of $L_1^N-L_0^N$, over all initial conditions. For any $R<\infty$,
it then follows from Theorem \ref{theo:convergence to stationary distribution for fixed N} and the ergodic theorem for Markov chains that
\[
\frac{1}{m}\sum_{k=0}^{m-1}[(L^N_{k+1}-L^N_k)\wedge R]
\]
has an almost sure limit as $m\ra\infty$ which does not depend upon the initial condition. It then follows from Birkhoff's theorem that the almost sure limit is given by
\[
\frac{1}{m}\sum_{k=0}^{m-1}[(L^N_{k+1}-L^N_k)\wedge R]\ra \expE_{\vec{Y}_0\sim \psi^N}[(L^N_1-L^N_0)\wedge R]\quad\text{almost surely as $m\ra\infty$,}
\]
for any initial condition. This almost sure limit is at most $\sqrt{2}$ by \eqref{eq:lim sup of velocity at most sqrt 2}. It follows that 
\[
\expE_{\vec{Y}_0\sim \psi^N}[(L^N_1-L^N_0)\wedge R]\leq \sqrt{2}\quad\text{for any}\quad R<\infty,
\]
whence we obtain by the monotone convergence theorem that 
\[
\expE_{\vec{Y}_0\sim \psi^N}[L^N_1-L^N_0]\leq \sqrt{2}.
\]
We can therefore repeat the above argument without $R$ to see that 
\[
\frac{L^N_m-L^N_0}{m}=\frac{1}{m}\sum_{k=0}^{m-1}(L^N_{k+1}-L^N_k)\ra \expE_{\vec{Y}_0\sim \psi^N}[L^N_1-L^N_0]\quad\text{almost surely as $m\ra\infty$,}
\]
for any initial condition. We define 
\[
v_N:=\expE_{\vec{Y}_0\sim \psi^N}[L^N_1-L^N_0]\leq \sqrt{2}.
\]

We now use this to obtain \eqref{eq:rightmost minus leftmost integrable}. There is a probability bounded away from $0$ that in time $1$ there are $N$ particle deaths, in each of which the minimal particle jumps onto the rightmost particle at that time, and that moreover $\sup_{1\leq i\leq N}\sup_{0\leq s,t\leq 1}\lvert W^i_s-W^i_t\rvert\leq 1$, where $W^1,\ldots,W^N$ are the driving Brownian motions. By considering this event, we see that if it were the case that $\expE_{\vec{Y}^N\sim \psi^N}[\max_{1\leq i\leq N}Y^{i}]=\infty$, then we would have $\expE_{\vec{Y}_0\sim \psi^N}[L^N_1-L^N_0]=+\infty$, which is a contradiction. Thus we have established \eqref{eq:rightmost minus leftmost integrable}.

We have established almost sure convergence in \eqref{eq:almost sure limit of lower boundary} along integer times. We now extend this to almost sure convergence along times in $\Rm_{\geq 0}$, using a Borel-Cantelli argument. 

We firstly observe that for any fixed $N<\infty$,
\begin{equation}\label{eq:expected value of movement of leftmost in time 1 bounded}
    \expE_{\vec{Y}^N_0\sim \psi^N}[\sup_{0\leq s<1}\lvert L^N_s-L^N_0\rvert]<\infty,
\end{equation}
by \eqref{eq:rightmost minus leftmost integrable} and the observation that $\sup_{0\leq s<1}\lvert L^N_s-L^N_0\rvert \le \max_i Y_i(0) + \sum_{i=1}^N \sup_{0\leq s<1}\lvert W^i_s\rvert$. We now observe that for any $\epsilon>0$, $m\in \mathbb{Z}_{\geq 0}$ and initial condition $\vec{y}_0\in \Gamma_N$, 
\[
\Pm_{\vec{y}_0}(\sup_{0\leq s<1}\lvert L^N_{m+s}-L^N_m\rvert >\epsilon m)\leq \lvert\lvert \Law_{\vec{y}_0}(\vec{Y}^N_m)-\psi^N\rvert\rvert_{\TV}+\Pm_{\vec{Y}_0\sim \psi^N}(\sup_{0\leq s<1}\lvert L^N_{s}-L^N_s\rvert >\epsilon m).
\]
The two terms on the right hand side are summable over $m\in \mathbb{Z}_{\geq 0}$ by Theorem \ref{theo:convergence to stationary distribution for fixed N} and \eqref{eq:expected value of movement of leftmost in time 1 bounded} respectively. It then follows from the Borel-Cantelli lemma that
\[
\sup_{0\leq s<1}\Big\lvert \frac{L^N_{m+s}}{m+s}-\frac{L^N_m}{m}\Big\rvert\ra 0\quad\text{almost surely as $m\ra\infty$,}
\]
for any initial condition. This concludes the proof.
\end{proof}

\subsection{Tightness of the stationary empirical measures}

\begin{defin}[Weak convergence in $\calP(\calP(\Rm))$ and $\calP(\calP(\Rm_{\geq 0}))$]\label{defin:weak convergence of random measures}
We recall that $\calP(\Rm)$ is equipped with the topology of weak convergence of measures, which is metrisable (say with the $\Wah$ or the Levy-Prokhorov metric). The space $\calP(\calP(\Rm))$ is then equipped with the topology of weak convergence of measures, where the underlying space is the metrisable space $\calP(\Rm)$ (note that the resultant topology on $\calP(\calP(\Rm))$ is agnostic to the choice of metric with which we metrise the topology on $\calP(\Rm)$). This then defines the notion of weak convergence in $\calP(\calP(\Rm))$. Since $\Rm$ is Polish, so too is $\calP(\Rm)$ and hence $\calP(\calP(\Rm))$. The notion of tightness in $\calP(\calP(\Rm))$ is therefore clear, and is equivalent to pre-compactness in $\calP(\calP(\Rm))$ by Prokhorov's theorem. We may replace $\Rm$ with $\Rm_{\geq 0}$, without any necessary changes.
\end{defin}

We define the stationary empirical measure $\chi^N$ by
\[
\chi^N:=\Theta^N_{\#}\psi^N=\Law_{\vec{Y}^N\sim \psi^N}(\frac{1}{N}\sum_{i=1}^N\delta_{Y^i}).
\]

We now establish the following proposition.
\begin{prop}\label{prop:chiN tight}
We have that $\expE_{\vec{Y}^N\sim \psi^N}[b(\vec{Y})]=v_N$ for all $N<\infty$. In particular, we have that $\{\chi^N:N<\infty\}$ is tight in $\calP(\calP(\Rm_{\geq 0}))$.
\end{prop}

\begin{proof}[Proof of Proposition \ref{prop:chiN tight}]

We firstly note that it immediately follows from \eqref{eq:rightmost minus leftmost integrable} that 
\[
\expE_{\vec{Y}^N\sim\psi^N}[b(\vec{Y}^N)]<\infty
\]
for all $N<\infty$.

We consider the $N$-BBM started from $\vec{Y}^N_0\sim \psi^N$. We denote the barycentre by
\begin{equation}\label{eq:barycentre}
M^N_t:=\frac{1}{N} \sum_{i=1}^N X^i_t = L_t^N+ b (\vec{Y}_t^N).
\end{equation}
It follows from \eqref{eq:vN equal to expectation} that
\begin{equation}\label{eq:expected change in M1N under stationarity}
\expE_{\vec{Y}^N_0\sim \psi^N}[M^N_1-M^N_0]=\expE_{\vec{Y}^N_0\sim \psi^N}[M^N_1-L^N_1]+\expE_{\vec{Y}^N_0\sim \psi^N}[L^N_1-L^N_0]-\expE_{\vec{Y}^N_0\sim \psi^N}[M^N_0-L^N_0]= v_N.
\end{equation}
The driving Brownian 
motion of particle $X^i_t$ is $W^i_t$, with $W^i_0:=0$. We recall that $\tau_n$ is the $n^{\text{th}}$ killing jump time of any particle, and  $\tau_n^i$ is the $n^{\text{th}}$ killing time of particle $X^i$. Then we have that
\begin{equation}\label{eq:expectation of movement of barycentre in terms of BMs and jumps}
M^N_t-M^N_0=\frac{1}{N}\sum_{i=1}^N(W^i_t-W^i_0)+\frac{1}{N}\sum_{i=1}^N\sum_{0<\tau^i_n\leq t}(X^{i}_{\tau^i_n}-X^i_{\tau^i_n-}).
\end{equation}

We note the following, which shall be employed later when we come to prove Proposition \ref{prop:characterisation of subsequential limits}.
\begin{rmk}\label{rmk:uniform integrability of barycentre from the left}
Since particles only jump to the right, we have that $M^N_t-M^N_0\geq\frac{1}{N}\sum_{i=1}^N(W^i_t-W^i_0)$,  from which it is easy to see that $(M^N_t-M^N_0)_-:=\lvert (M^N_t-M^N_0)\wedge 0\rvert $ is uniformly integrable, over all initial conditions and all $N<\infty$. 
\end{rmk}
We take the expectation of \eqref{eq:expectation of movement of barycentre in terms of BMs and jumps} with $t=1$ and $\vec{Y}^N_0\sim \psi^N$ to see that
\begin{equation}\label{eq:expectation for vN sum of jumps}
v_N=\expE_{\vec{Y}^N_0\sim \psi^N}\Big[\frac{1}{N}\sum_{i=1}^N\sum_{0<\tau^i_n\leq 1}(X^{i}_{\tau^i_n}-X^i_{\tau^i_n-})\Big].
\end{equation}

We now claim that
\begin{equation}\label{eq:mg 1 tightness pf}
\frac{1}{N}\sum_{i=1}^N\sum_{0<\tau^i_n\leq t}(X^{i}_{\tau^i_n}-X^i_{\tau^i_n-})-\int_0^tb(\vec{Y}^N_s)ds\quad\text{is a local martingale.}
\end{equation}
Observe that at time $\tau_n^i-$  (just before $X^i$ is killed and jumps for $n^{\text{th}}$ time) the particle $X^i$ is necessarily the leftmost. Thus, given the configuration at time $\tau_n^i-$, the expected value of $X^i_{\tau^i_n}-X^i_{\tau^i_n-}$ is precisely $b(\vec{Y}^n_{s-})$. It follows that 
\begin{equation}\label{eq:mg 2 tightness pf}
\frac{1}{N}\sum_{i=1}^N\sum_{0<\tau^i_n\leq t}(X^{i}_{\tau^i_n}-X^i_{\tau^i_n-})-\frac{1}{N}\sum_{0<\tau_n\leq t}b(\vec{Y}^N_{\tau_n-})\quad\text{is a local martingale.}
\end{equation}

We recall that $N_t:=\sup\{n:\tau_n\leq t\}$ is the number of jumps up to time $t$, so that $N_t-Nt$ is a martingale. Then
\begin{equation}\label{eq:mg 3 tightness pf}
\frac{1}{N}\sum_{0<\tau_n\leq t}b(\vec{Y}^N_{\tau_n-})-\int_0^tb(\vec{Y}^N_s)ds=\frac{1}{N}\int_0^tb(\vec{Y}^N_{s-})d(N_s-Ns)\quad\text{is a local martingale.}
\end{equation}
Combining \eqref{eq:mg 2 tightness pf} and \eqref{eq:mg 3 tightness pf}, we obtain \eqref{eq:mg 1 tightness pf}.

We now define $(T_k)_{k=1}^{\infty}$ to be a sequence of stopping times reducing the local martingale in \eqref{eq:mg 1 tightness pf} (so that in particular $T_k\uparrow +\infty$ almost surely), which without loss of generality we assume to be non-decreasing (i.e. $T_1\leq T_2\leq\ldots$ almost surely). We see that for all $k<\infty$,
\[
\expE_{\vec{Y}^N_0\sim\psi^N}\Big[\frac{1}{N}\sum_{i=1}^N\sum_{0<\tau^i_n\leq T_k\wedge 1}(X^{i}_{\tau^i_n}-X^i_{\tau^i_n-})\Big]=\expE_{\vec{Y}^N_0\sim\psi^N}\Big[\int_0^{T_k\wedge 1}b(\vec{Y}^N_s)ds\Big].
\]
It then follows from the monotone convergence theorem, \eqref{eq:expectation for vN sum of jumps} and Tonelli's theorem that
\begin{align*}
v_N&=\expE_{\vec{Y}^N_0\sim\psi^N}\Big[\frac{1}{N}\sum_{i=1}^N\sum_{0<\tau^i_n\leq 1}(X^{i}_{\tau^i_n}-X^i_{\tau^i_n-})\Big]\overset{\text{MCT}}{=}\lim_{k\ra\infty}\expE_{\vec{Y}^N_0\sim\psi^N}\Big[\frac{1}{N}\sum_{i=1}^N\sum_{0<\tau^i_n\leq T_k\wedge 1}(X^{i}_{\tau^i_n}-X^i_{\tau^i_n-})\Big]\\
&=\lim_{k\ra\infty}\expE_{\vec{Y}^N_0\sim\psi^N}\Big[\int_0^{T_k\wedge 1}b(\vec{Y}^N_s)ds\Big]
\overset{\text{MCT}}{=}\expE_{\vec{Y}^N_0\sim \psi^N}\Big[\int_0^1b(\vec{Y}^N_s)ds\Big]\overset{\text{Tonelli}}{=}\expE_{\vec{Y}\sim \psi^N}[b(\vec{Y})].
\end{align*}

It follows that the mean measure 
\begin{equation}\label{eq:mean measures}
\xi^N(\cdot):= \frac{1}{N} \expE_{\vec{Y}^N_0\sim\psi^N} \Big[\sum_{i=1}^N\delta_{Y^{i}}(\cdot)\Big]
\end{equation}
has first moment $v_N\leq \sqrt{2}$. In particular, $\{\xi^N:N<\infty\}$ is a tight family of measures in $\calP(\Rm_{\geq 0})$. Given a complete, separable metric space $S$ and $K\subseteq \calP(\calP(S))$, \cite[Theorem 4.10]{Kallenberg2017} ensures that $K$ being tight in $\calP(\calP(S))$ is equivalent to the corresponding mean measures being tight in $\calP(S)$. It follows that $\{\chi^N:N<\infty\}$ is tight in $\calP(\calP(\Rm_{\geq 0}))$. 
\end{proof}

\subsection{Centring by the median}

For $\mu\in\calP(\Rm)$ we define
\begin{equation}\label{eq:lower bdy and median maps} 
\begin{aligned}
    A(\mu)&:=\inf\{x\in \Rm:\mu ([x,\infty))<\frac{1}{2}\},\quad L(\mu):=\inf\{x:\mu([x,\infty))<1\},\\ H(\mu)&:=\int_{\Rm}\lvert x\rvert \mu (dx),\quad M(\mu):=\int_{\Rm}x\mu(dx),
\end{aligned}
\end{equation}
where $L(\mu):=-\infty$ when there is no such $x$, $H(\mu)$ is possibly $+\infty$, and $M(\mu)$ is defined only when $H(\mu)$ is finite. We claim that 
\begin{equation}\label{eq:A and L upper semicontinuous}
    A:\calP(\Rm)\ra \Rm\quad\text{and}\quad L:\calP(\Rm)\ra \Rm\cup\{-\infty\}\quad\text{are upper semicontinuous.}
\end{equation}
\begin{proof}[Proof of \eqref{eq:A and L upper semicontinuous}]
We take an arbitrary sequence $(\mu_n)_{n=1}^{\infty}$ in $\calP(\Rm)$ converging weakly to $ \mu\in \calP(\Rm)$ as $n\ra\infty$. We observe that $\limsup_{n\ra\infty}\mu_n([x,\infty))<\frac{1}{2}$ for any $x>A(\mu)$ and $\limsup_{n\ra\infty}\mu_n([x,\infty))<1$ for any $x>L(\mu)$, whence the claim follows.
\end{proof}
We define 
\begin{equation}\label{eq:calPfdefin}
\calP_f(\Rm):=\{\mu\in \calP(\Rm):L(\mu)>-\infty\}. 
\end{equation}
This is a measurable subset of $\calP(\Rm)$ by \eqref{eq:A and L upper semicontinuous}. We observe, by Markov's inequality, that
\begin{equation}\label{eq:inequality between median and mean}
A(\mu)-L(\mu)\leq 2(M(\mu)-L(\mu))\quad\text{for all $\mu\in \calP_f(\Rm)$.}
\end{equation}

Given a solution $(L_t,u_t)$ of the PDE \eqref{eq:hydrodynamic limit}, with initial condition $\mu$, we define $\mu_t(dx)=u_t(x)dx$, $A_t(\mu):=A(\mu_t)$ and $L_t(\mu)=L(\mu_t)=L_t$. Then we define 
\begin{equation}\label{eq:shifted flow}
\tilde{u}(x,t):=u(x+A_t(\mu),t),
\end{equation}
so that the median is fixed at $0$. This provides the hydrodynamic limit for the $N$-BBM centred by its median particle. It defines a flow on
\begin{equation}
\calP_{c}(\Rm):=\{\mu\in \calP(\Rm):A (\mu)=0\},
\end{equation}
which we denote by $\Phi_t$ (note that $\calP_c(\Rm)$ is a measurable subset of $\calP(\Rm)$ by \eqref{eq:A and L upper semicontinuous}). That is, $\Phi_t(\mu) = \tilde{u}(x,t)$ given by \eqref{eq:shifted flow}, where $u$ is the solution of  \eqref{eq:hydrodynamic limit} at time $t$ with initial condition $\mu$, centred by its median. We extend the definition of $\Phi_t(\mu)$ for $t>0$ to all $\mu\in \calP(\Rm)$, without requiring that $A(\mu)=0$ (although $\Phi_t(\mu)\in \calP_c(\Rm)$ for all $\mu\in \calP(\Rm)$ and $t>0$). 

We define the map
\begin{equation}
\alpha:\calP(\Rm)\ni \mu\mapsto \alpha(\mu):=(x\mapsto (x-A(\mu))_{\#}\mu \in \calP_c(\Rm),
\end{equation}
which takes a probability measure and recentres it so that its median is at $0$. We note that $A:\calP(\Rm)\ra \Rm$ is measurable by \eqref{eq:A and L upper semicontinuous}, from which we deduce that
\begin{equation}\label{eq:measurability of alpha}
\alpha:\calP(\Rm)\ra \calP_c(\Rm)\quad\text{is measurable.}
\end{equation}
It follows from this and Lemma \ref{lem:sensitivity depending on ics} that $\Phi_t:\calP(\Rm)\ra \calP_c(\Rm)$ is measurable for all $t>0$.

Suppose that $(u,L)$ is a solution of \eqref{eq:hydrodynamic limit} with initial condition $\mu$ such that $H(\mu)<\infty$. Consider the stochastic representation of $u$ provided for in Appendix \ref{appendix:stochastic representation for PDE}: by Theorem \ref{theo:stochastic representation for PDE}, $u(x,t)dx=e^t \mathbb{P}_{\mu}(B_t\in dx, \tau>t)\leq \mathbb{P}_{\mu}(B_t\in dx)$, where $B_t$ is a Brownian motion and $\tau$ is the first hitting time of $L$ by $B_t$. It follows that $H(\Phi_t(\mu))<\infty$ for all $t>0$. 

Therefore $\Phi_t$ defines, by restriction, a flow on
\begin{equation}
\calP_{c,1}(\Rm):=\{\mu\in \calP(\Rm):A (\mu)=0,H(\mu)<\infty\}.
\end{equation}
For $\mu\in \calP_{c,1}(\Rm)$ we may therefore define $M_t(\mu):=M(\mu_t)$ and $B_t(\mu):=H(\mu_t)$ for all $t\geq 0$.

We define, by abuse of notation, $A(\vec{x}):=A(\frac{1}{N}\sum_{i=1}^N\delta_{x^i})$ for $\vec{x}=(x^1,\ldots,x^N)\in \Rm^N$, giving the median particle. Given an $N$-BBM $\vec{X}^N_t=(X^1_t,\ldots,X^N_t)$, we define 
\[
Z^i_t=X^i_t-A(\vec{X}^N_t)\quad\text{and}\quad \vec{Z}^N_t:=(Z^1_t,\ldots,Z^N_t), 
\]
 the $N$-BBM centred by its median particle. The state space of $\vec{Z}^N_t$ is then $\mathbb{A}_N:=\{\vec{z}\in \Rm^N:A(\vec{z})=0\}$. The unique stationary distribution of $\vec{Z}^N_t$ is then given by 
\begin{equation}\label{eq:stationary dist median centred}
\varphi^N:=\Law_{\vec{Y}\sim \psi^N}((Y^1-A(\vec{Y}),\ldots,Y^N-A(\vec{Y}))).
\end{equation} 

The stationary empirical measure for $(Z^N_t)_{t\geq 0}$ is then given by
\begin{equation}\label{eq:median centred stationary empirical measure}
\zeta^N:=\alpha_{\#}\chi^N=\Theta^N_{\#}\ \varphi^N=\Law_{\vec{Z}\sim \varphi^N}\Big(\frac{1}{N}\sum_{i=1}^N\delta_{Z^i}\Big).
\end{equation}
Since $\expE_{\mu\sim \chi^N}[M(\mu)-L(\mu)]=\expE_{\vec{Y}^N\sim \psi^N}[b(\vec{Y}^N)]=v_N$ (by Proposition \ref{prop:chiN tight}), it follows from \eqref{eq:inequality between median and mean} that for all $N<\infty$ we have
\begin{equation}\label{eq:expected value of difference between median and minima}
\begin{split}
&\expE_{\mu\sim \zeta^N}[-L(\mu)]=\expE_{\mu\sim \zeta^N}[A(\mu)-L(\mu)]=\expE_{\mu\sim \chi^N}[A(\mu)-L(\mu)]\\
&\leq 2\expE_{\mu\sim \chi^N}[M(\mu)-L(\mu)]=2v_N\leq 2\sqrt{2}.
\end{split}
\end{equation}
We now establish the following proposition.
\begin{prop}\label{prop:zetaN tight and subseq limit of zetaN supported on Pf}
We have that $\{\zeta^N:N<\infty\}$ is tight in $\calP(\calP(\Rm))$. Moreover every subsequential limit in $\calP(\calP(\Rm))$ of $\{\zeta^N:N<\infty\}$ is supported on $\calP_f(\Rm)$.
\end{prop}
\begin{proof}[Proof of Proposition \ref{prop:zetaN tight and subseq limit of zetaN supported on Pf}]
The tightness of $\{\chi^N<\infty\}$ is given by Proposition \ref{prop:chiN tight}. 

We deduce the the tightness of $\{\zeta^N:N<\infty\}$ as follows. We fix $\epsilon>0$. Since $\{\chi^N:N<\infty\}$ is tight in $\calP(\calP(\Rm_{\geq 0}))$, there exists compact $K\subseteq \calP(\Rm_{\geq 0})$ such that $\chi^N(K)>1-\epsilon$ for all $N<\infty$. Then $A(\mu)\geq 0$ for all $\mu\in K$. Since $A$ is upper semicontinuous by \eqref{eq:A and L upper semicontinuous}, it is bounded from above, so there exists compact $I\subseteq \Rm_{\geq 0}$ such that $A(\mu)\in I$ whenever $\mu\in K$. 

We now define $K_I:=\{(x\mapsto (x-y))_{\#}\mu:\mu\in K, y\in I\}$ - i.e. the set of measures in $K$ translated by the negative of a number in $I$. We observe that $K_I$ is pre-compact by Prokhorov's theorem, and that if $\mu\in K$ then $\alpha(\mu)\in K_I$, since  $A(\mu)\in I$.  It follows that 
\[
\zeta^N(\text{cl}(K_I))\geq \zeta^N(K_I)\geq \chi^N(K)\geq 1-\epsilon
\]
for all $N<\infty$. Since $\epsilon>0$ is arbitrary, we conclude that $\{\zeta^N<\infty\}$ is tight.

We now take $\zeta$ to be some subsequential limit of $\zeta^N$. Applying Skorokhod's representation theorem, there exists for this subsequence a collection of random measures $(\mu_N:N<\infty)$ and $\mu$ such that $\mu_N\ra \mu$ $\tilde{\Pm}$-almost surely in Wasserstein distance $\Wah$ as $N\ra\infty$, on some probability space. We have from \eqref{eq:expected value of difference between median and minima} that $(-L(\mu_N):1\leq N<\infty)$ is a sequence of non-negative, integrable random variables with expectation $\tilde{\expE}[-L(\mu_N)]\leq 2\sqrt{2}$. Since 
\[
0\leq -L(\mu)\leq \liminf_{N\ra\infty}(-L(\mu_N)),
\]
it follows from Fatou's lemma that 
\[
\tilde{\expE}[-L(\mu)]\leq \tilde{\expE}[\liminf_{N\ra\infty}(-L(\mu_N))] \leq \liminf_{N\ra\infty}\tilde{\expE}[-L(\mu_N)]\leq  2\sqrt{2},
\]
so that $0\leq -L(\mu)<\infty$ almost surely. Therefore $\zeta$ is supported on $\calP_f(\Rm)$.
\end{proof}

\subsection{Characterisation of subsequential limits}

We have established in Proposition \ref{prop:zetaN tight and subseq limit of zetaN supported on Pf} that $\{\zeta^N:N<\infty\}$ is tight in $\calP(\calP(\Rm))$, and that every subsequential limit is supported on $\calP_f(\Rm)$. We now establish the following proposition.
\begin{prop}\label{prop:characterisation of subsequential limits}
For any subsequential limit, $\zeta$, of $\zeta^N$, we have that $\zeta$ is a stationary measure under $(\Phi_t)_{t\geq 0}$, so that $(\Phi_t)_{\#}\zeta=\zeta$ for all $t\geq 0$, and moreover
\begin{equation}\label{eq:expected increase in Lt at most sqrt 2}
    \expE_{\mu\sim \zeta}[L_t(\mu)-L_0(\mu)]\leq \sqrt{2}t\quad \text{for all}\quad t\geq 0.
\end{equation}
\end{prop}

\begin{proof}[Proof of Proposition \ref{prop:characterisation of subsequential limits}]
De Masi et al. \cite[Theorem 1]{DeMasi2019} have provided a hydrodynamic limit theorem for the $N$-BBM. However, they assumed that initially the particles are independent and identically distributed according to a distribution belonging to $\calP_f(\Rm)$ which has an absolutely continuous density. On the other hand, we will need the following hydrodynamic limit theorem, applicable when the initial empirical measures are assumed only to be a tight family of random measures with subsequential limits supported on $\calP_f(\Rm)$.

Given an $N$-BBM $(\vec{X}^N_t)_{t\geq 0}$ we write $m^N_t$ for the empirical measure
\begin{equation}\label{eq:mN empirical measure without recentring}
m^N_t:=\Theta^N(\vec{X}^N_t),\quad t\geq 0.
\end{equation}
Given $\mu\in\calP(\Rm)$ we write $\Psi_t(\mu)$ for the solution at time $t$ of \eqref{eq:hydrodynamic limit} with $u_0 = \mu$. This then defines a flow on $\calP_f(\Rm)$ (and on $\calP(\Rm)$), denoted by $(\Psi_t)_{t\geq 0}$ (note that $\Psi_t$ is just $\Phi_t$ without centring by the median). We note that $\Psi_t:\calP(\Rm)\ra \calP_f(\Rm)$ is continuous (hence measurable, in particular) by Lemma \ref{lem:sensitivity depending on ics}. 

We recall here that the notion of convergence and tightness in $\calP(\calP(\Rm))$ was defined in Definition \ref{defin:weak convergence of random measures}.

\begin{theo}\label{theo:hydrodynamic limit for random ics}
We suppose that we have a sequence of $N$-BBMs $(\vec{X}^N_t)_{t\geq 0}$ 
such that $\Law(m^N_0)\in \calP(\calP(\Rm))$ converges  
to $\Lambda\in \calP(\calP(\Rm ))$, and assume that $\Lambda(\calP_f(\Rm))=1$. Then $\Law((m^N_0,m^N_t))$ converges 
to $\Law_{\mu\sim \Lambda}(\mu,\Psi_t(\mu))$, for any fixed $t>0$.
\end{theo}
Theorem \ref{theo:hydrodynamic limit for random ics} is established in the appendix.

We now consider a sequence of $N$-BBMs, $(\vec{X}^N_t)_{t\geq 0}$, such that $\vec{X}^N_0\sim \varphi^N$, the stationary distribution of the $N$-BBM with its median fixed at $0$ (see \eqref{eq:stationary dist median centred}). We note, however, that $\vec{X}^N_t$ evolves as the $N$-BBM without recentring, so that the median may be non-zero at positive times. Recall that 
\[
M^N_t=M(m^N_t),\quad L^N_t=L(m^N_t), \text{ and write }  A^N_t:=A(m^N_t)
\]
for $t\geq 0$. 

We now fix arbitrary $t>0$. By Proposition  \ref{prop:zetaN tight and subseq limit of zetaN supported on Pf} the sequence $\zeta^N$ is tight and $m^N_0 \sim \zeta^N$. Thus we can choose a subsequence along which $\zeta^N \to \zeta$, for some subsequential limit $\zeta$. Using \eqref{eq:expected value of difference between median and minima} and the stationarity of $\chi^N$, we have that $A^N_0$, $M^N_0$, $L^N_0$ $M^N_t-L^N_t$ and $A^N_t-L^N_t$ are all tight. It also follows from \eqref{eq:expected change in M1N under stationarity} that $M^N_t-M^N_0$ is tight, so $A^N_t$, $M^N_t$ and $L^N_t$ are also all tight. Then by Theorem \ref{theo:hydrodynamic limit for random ics} and Skorokhod's representation theorem, we can find a further subsequence and new probability space $(\tilde{\Omega},\tilde{\mathcal{F}},\tilde{\Pm})$ (we write $\tilde{\expE}$ for expectation with respect to $\tilde{\Pm}$) along which
\[
m^N_0\ra m_0\quad\text{and}\quad m^N_t\ra m_t\quad\text{$\tilde{\Pm}$-almost surely in $\Wah$}
\]
as $N\ra\infty$, whereby $m_t=\Psi_t(m_0)$ and $m_0\sim \zeta$, and along which
\[
\begin{aligned}
A^N_0\overset{\text{$\tilde{\Pm}$-a.s.}}{\ra} \tilde{A}_0,\quad A^N_t\overset{\text{$\tilde{\Pm}$-a.s.}}{\ra} \tilde{A}_t,\quad M^N_0\overset{\text{$\tilde{\Pm}$-a.s.}}{\ra} \tilde{M}_0,\quad  M^N_t\overset{\text{$\tilde{\Pm}$-a.s.}}{\ra} \tilde{M}_t,\quad L^N_0\overset{\text{$\tilde{\Pm}$-a.s.}}{\ra} \tilde{L}_0,\quad L^N_t\overset{\text{$\tilde{\Pm}$-a.s.}}{\ra} \tilde{L}_t,
\end{aligned}
\]
as $N\ra\infty$, the limits on the right being random variables whose identities have not yet been determined. 

We now prove that $\zeta$ must be a stationary measure for $(\Phi_t)_{t\geq 0}$ (the flow of \eqref{eq:hydrodynamic limit} with recentring by the median). We write $k^N_t$ and $k_t$ for $m^N_t$ and $m_t$ (respectively) shifted so that they are centred at their median,
\[
k^N_t:=\alpha(m^N_t)\quad\text{and}\quad k_t:=\alpha(m_t).
\]
The cumulative distribution function (CDF) of $m_t(\omega)$ is strictly decreasing to the right of $L(m_t(\omega))$ (by Remark \ref{rmk:U strictly decreasing at pve times}) and continuous on all of $\Rm$ (by \cite[Theorem 1.1]{Berestycki2018}), for all $\omega\in \tilde{\Omega}$. We have $\Wah$ convergence of $m_t^N(\omega)$ to $m_t(\omega)$ for $\tilde{\Pm}$-almost every $\omega\in\tilde{\Omega}$. This implies convergence in distribution, hence pointwise convergence of the CDF (since the limiting CDF is continuous), which then implies convergence of the quantiles (except, perhaps, for the $0$ and $1$ quantiles) since the CDF of $m_t(\omega)$ is strictly decreasing to the right of $L(m_t(\omega))$. It follows in particular that $A(m^N_t)\ra A(m_t)$ as $N\ra\infty$ $\tilde{\Pm}$-almost surely, so that
\begin{equation}
    \tilde{A}_t=A(m_t)\quad\text{for}\quad t>0.
\end{equation} 
Then it follows that 
\begin{equation}
    k^N_t\ra k_t\quad\text{in}\quad \Wah\quad\text{$\tilde{\Pm}$-almost surely.}
\end{equation}

Since
\[
\Law(k^N_t)=\zeta^N\quad\text{and}\quad\Law(k_t)=\Law_{\mu\sim \zeta}(\Phi_t(\mu))=(\Phi_t)_{\#}\zeta,
\]
it follows that 
\[
\zeta^N\ra (\Phi_t)_{\#}\zeta
\]
in $\calP(\calP(\Rm))$ as $N\ra\infty$, so that $(\Phi_t)_{\#}\zeta=\zeta$. Since $t>0$ was arbitrary, it follows that $\zeta$ is an invariant measure for the flow of $(\Phi_t)_{t\geq 0}$.

We now turn to the proof of \eqref{eq:expected increase in Lt at most sqrt 2}. We recall the following:
\begin{enumerate}
    \item Using \eqref{eq:inequality between median and mean}, we have for all $N<\infty$ that 
\[
\begin{aligned}
&L_0^N\leq A_0^N\wedge M_0^N,\quad A_0^N=0,\quad L_t^N\leq A_t^N\wedge M_t^N,\\ &A_0^N-L_0^N\leq 2(M_0^N-L_0^N),\quad A_t^N-L_t^N\leq 2(M_t^N-L_t^N).
\end{aligned}
\]
\item By Remark \ref{rmk:uniform integrability of barycentre from the left}, we have that $(M^N_t-M^N_0)_-:=\lvert (M^N_t-M^N_0)\wedge 0\rvert $ is uniformly integrable over all $N<\infty$.
\item We have for all $N<\infty$ that
\[
\tilde{\expE}[M_0^N-L_0^N], \tilde{\expE}[M_t^N-L_t^N]\leq v_N\leq \sqrt{2}\quad\text{and}\quad \tilde{\expE}[M_t^N-M_0^N]= v_Nt\leq \sqrt{2}t.
\]
\end{enumerate}
Using Fatou's lemma, we therefore have the following:
\begin{enumerate}
    \item We have that 
\[
\begin{split}
    & \tilde{L}_0\leq \tilde{A}_0\wedge \tilde{M}_0,\quad \tilde{A}_0=0,\quad \tilde{L}_t\leq \tilde{A}_t\wedge \tilde{M}_t,\\& \tilde{A}_0-\tilde{L}_0\leq 2(\tilde{M}_0-\tilde{L}_0),\quad \tilde{A}_t-\tilde{L}_t\leq 2(\tilde{M}_t-\tilde{L}_t).
\end{split}
\]
\item $\tilde{M}_t-\tilde{M}_0$ is integrable with $\tilde{\expE}[\tilde{M}_t-\tilde{M}_0]\leq t\sqrt{2}$.
\item We have that $\tilde{\expE}[\tilde{M}_0-\tilde{L}_0], \tilde{\expE}[\tilde{M}_t-\tilde{L}_t]\leq \sqrt{2}$, so in particular $\tilde{M}_t-\tilde{L}_t$ and $\tilde{M}_0-\tilde{L}_0$ are non-negative, integrable random variables.
\item It therefore follows that $\tilde{A}_t-\tilde{L}_t$ and $\tilde{A}_0-\tilde{L}_0$ are also non-negative, integrable random variables with $\tilde{\expE}[\tilde{A}_0-\tilde{L}_0],\tilde{\expE}[\tilde{A}_t-\tilde{L}_t]\leq 2\sqrt{2}$.
\end{enumerate}

We now observe that
\[
\begin{aligned}
    \tilde{A}_t=\underbrace{(\tilde{A}_t-\tilde{L}_t)}_{ \leq 2(\tilde{M}_t-\tilde{L}_t)}+\underbrace{(\tilde{L}_t-\tilde{M}_t)}_{\leq 0}+(\tilde{M}_t-\tilde{M}_0)+(\tilde{M}_0-\tilde{L}_0)+\underbrace{(\tilde{L}_0-\tilde{A}_0)}_{\leq 0}+\underbrace{\tilde{A}_0}_{=0}.
\end{aligned}
\]
It follows that $\tilde{A}_t$ is an integrable random variable with $\tilde{\expE}[\tilde{A}_t]\leq (3+t)\sqrt{2}$. Since $\tilde{A}_t=A(m_t)$, it follows that $A(m_t)$ is integrable with
\begin{equation}\label{eq:expectation of At}
    \tilde{\expE}[A(m_t)]\leq (3+t)\sqrt{2}.
\end{equation}


Since $m_0^N\ra m_0$ in $\Wah$ $\tilde{\Pm}$-almost surely and $A:\calP(\Rm)\ra \Rm$ is upper semicontinuous \eqref{eq:A and L upper semicontinuous}, we have
\begin{equation}\label{eq:A0 nonnegative}
A(m_0)\geq \limsup_{N\ra\infty}A(m^N_0)=\tilde{A}_0=0\quad\text{$\tilde{\Pm}$-almost surely.}
\end{equation}

It follows from \eqref{eq:expectation of At} and \eqref{eq:A0 nonnegative} that the positive part of $A(m_t)-A(m_0)$, $(A(m_t)-A(m_0))_+:=(A(m_t)-A(m_0))\vee 0$, is $\tilde{\Pm}$-integrable. Its expectation with respect to $\tilde{\Pm}$ is therefore well-defined, possibly taking the value $-\infty$. In particular, \eqref{eq:expectation of At} and \eqref{eq:A0 nonnegative} imply that
\begin{equation}\label{eq:change of median expectation upper bound}
\tilde{\expE}[A(m_t)-A(m_0)]\in [-\infty,(3+t)\sqrt{2}].
\end{equation}

Since $L:(\calP(\Rm),\Wah)\ra \Rm\cup \{-\infty\}$ is upper semicontinuous \eqref{eq:A and L upper semicontinuous}, we have
\[
L(m_t)\geq \limsup_{N\ra\infty}L(m_t^N)=\tilde{L}_t\quad \text{$\Pm$-almost surely.}
\]
We recall that $\tilde{A}_t=A(m_t)$ $\tilde{\Pm}$-almost surely. Therefore we have
\[
0\leq A(m_t)-L(m_t)\leq \tilde{A}_t-\tilde{L}_t,
\]
which is $\tilde{\Pm}$-integrable. Then since 
\[
A(m_0)-L(m_0)\overset{d}{=}A(m_t)-L(m_t),
\]
it follows that $A(m_t)-L(m_t)$ is also $\tilde{\Pm}$-integrable with the same expectation as $A(m_0)-L(m_0)$. It then follows from \eqref{eq:change of median expectation upper bound} that $(L(m_t)-L(m_0))_+:=(L(m_t)-L(m_0))\vee 0$ is $\tilde{\Pm}$-integrable, with
\begin{equation}
    \tilde{\expE}[L(m_t)-L(m_0)]\in [-\infty,(3+t)\sqrt{2}].
\end{equation}

Since $m_t=\Psi_t(m_0)$ and $L(m_0)>-\infty$ $\Pm$-almost surely, it follows from the comparison principle \cite[Theorem 1.2]{Berestycki2018} and the correspondence between solutions of \eqref{eq:hydrodynamic limit} and solutions of \eqref{eq:hydrodynamic limit CDF} that 
\[
L(m_t)-L(m_0)\geq L(\Psi_t(\delta_{L(m_0)}))-L(m_0) =L(\Psi_t(\delta_0)),
\]
$\Pm$-almost surely. In particular, $L(m_t)-L(m_0)$ is bounded from below away from $-\infty$. Therefore $L(m_t)-L(m_0)$ is an integrable random variable with
\begin{equation}
\tilde{\expE}[L(m_t)-L(m_0)]\leq (3+t)\sqrt{2}. 
\end{equation}

It therefore follows that
\[
\expE_{\mu\sim \zeta}[L_t(\mu)-L_0(\mu)]\leq (3+t)\sqrt{2}\quad \text{for all}\quad 0\leq t<\infty.
\]
We therefore have, for any $n\in \Nm$ and $t<\infty$, that
\[
\expE_{\mu\sim \zeta}[L_t(\mu)-L_0(\mu)]=\frac{1}{n}\expE_{\mu\sim \zeta}[L_{nt}(\mu)-L_0(\mu)]\leq \frac{(3+nt)}{n}\sqrt{2}.
\]
Since $n\in \Nm$ is arbitrary, we see that $\expE_{\mu\sim \zeta}[L_t(\mu)-L_0(\mu)]\leq t\sqrt{2}$.\end{proof}

\commentout{
we have that
\[
0\leq L(m_0)\leq A(m_0)\underbrace{\leq}_{\text{Markov}} 2M(m_0)\underbrace{\leq}_{\text{Fatou}} 2\liminf_{N\ra\infty}M^N_0=2\tilde{M}_0,
\]
which implies that $A(m_0)$ is a non-negative, integrable random variable. We therefore conclude that
\begin{equation}\label{eq:expected increase in A}
\tilde{\expE}[A(m_t)-A(m_0)]\leq (3+t)\sqrt{2}. 
\end{equation}

Since $0\leq L(m_0)\leq A(m_0)$, it follows that $A(m_0)-L(m_0)$ is an integrable random variable. 

It therefore follows from \eqref{eq:expected increase in A} that }

\subsection{Conclusion of the proof}\label{subsection:conclusion of proof}

We write $\tilde{\pi}_{\min}$ for $\pi_{\min}$ shifted so that its median is at $0$, $\tilde{\pi}_{\min}:=\alpha(\pi_{\min})$, so that $\int_0^{\infty}\tilde{\pi}_{\min}(x)dx=\frac{1}{2}$.

Our goal is to establish the following proposition.
\begin{prop}\label{prop:possible invariant measure unique}
If $\zeta$ satisfies the conclusions of Proposition \ref{prop:characterisation of subsequential limits}, then $\zeta=\delta_{\tilde{\pi}_{\min}}$.
\end{prop}

\begin{proof}[Proof of Proposition \ref{prop:possible invariant measure unique}]
The present proof shall hinge upon consideration of solutions of \eqref{eq:hydrodynamic limit CDF}. In particular, we will use the fact that solutions of \eqref{eq:hydrodynamic limit} and solutions of \eqref{eq:hydrodynamic limit CDF} are equivalent and the stationary measures for the corresponding flows are well-defined (in the language of dynamical systems the two flows are topologically conjugates). We therefore begin with some definitions and properties for solutions of \eqref{eq:hydrodynamic limit CDF}.

Consider the map $\kappa$ which takes a probability measure and maps it to its cumulative distribution function. More precisely, for $\mu \in \calP(\Rm)$ let
\[
\kappa (\mu)(x):= \mu((x,\infty)), \quad \text{for all}\quad  x\in \Rm.
\]
Note that $\kappa$ is a bijection $\kappa:\calP(\Rm)\ra \calD$, where 
\[
\calD:=\{U:\Rm\ra [0,1]:\text{$U$ non-increasing, $\lim _{x\ra-\infty} U(x) =1$ and $\lim_{x\ra \infty} U(x) = 0$, $U$ c\`adl\`ag}\}.
\]
It is a standard fact that $\mu_n\ra \mu$ weakly in measure ($\tau_{\calP(\Rm)}$) if and only if $\kappa(\mu_n)(x)\ra \kappa(\mu)(x)$ at all continuity points of $\kappa(\mu)$. This corresponds to the right topology to endow $\calD$ with.

Formally, $\kappa$ naturally induces a topology $\tau_{\calD}$ on $\calD$ which is the push forward of $\tau_{\calP(\Rm)}$. Therefore
\[
U_n\overset{\tau_{\calD}}{\ra} U\quad\text{if and only if}\quad U_n(x)\ra U(x)\quad\text{at all continuity points of $U(x)$.}
\]
It follows from the definitions that $\kappa:\calP(\Rm)\ra \calD$ is a homeomorphism.

Existence and uniqueness of solutions to \eqref{eq:hydrodynamic limit CDF} was established in \cite{Berestycki2018} for initial conditions $U_0$ belonging to $\calD$. Furthermore,  it is also shown in \cite{Berestycki2018} that $(u,L)$ is a solution of \eqref{eq:hydrodynamic limit} with initial condition $\mu_0\in \calP(\Rm)$ if and only if $(U,L)$ is a solution of \eqref{eq:hydrodynamic limit CDF} with initial condition $\kappa(\mu_0)$, where $U_t=\kappa(u_t)$ for $t>0$ and the free boundary $L_t$ is unchanged (here we identify $u_t$ with the measure having density $u_t$).

Given $U_0\in \calD$, we define $\Psi^{\CDF}_t(U_0)$ to be the solution of \eqref{eq:hydrodynamic limit CDF} at time $t$ with initial condition $U_0$ (without recentring). In particular we have that
\begin{equation}
\Psi^{\CDF}_t=\kappa\circ \Psi_t\circ \kappa^{-1}
\end{equation}
on $\calD$ and note that  by \cite{Berestycki2018}, for all $t>0$, $U_t = \Psi^{\CDF}_t(U_0)$ is in fact continuous.

Since $\kappa:\calP(\Rm)\ra (\calD,\tau_{\calD})$ is a homeomorphism and $\Psi_t:\calP(\Rm)\ra \calP(\Rm)$ is continuous (by Lemma \ref{lem:sensitivity depending on ics}),
\begin{equation}\label{eq:continuity of PsiCDF}
\Psi^{\CDF}_t:(\calD,\tau_{\calD})\ra (\calD,\tau_{\calD})\quad\text{is continuous for all $t > 0$}
\end{equation}
(and hence measurable in particular). 

We recall that $\Phi_t$ denotes the flow of \eqref{eq:hydrodynamic limit} with recentring so that the median is fixed at $0$. Then the conjugate flow
\[
\Phi_t^{\CDF}=\kappa\circ \Phi_t\circ \kappa^{-1}
\]
provides for the flow of \eqref{eq:hydrodynamic limit CDF} with recentring so that the $\frac{1}{2}$-level set is fixed at $0$. Since $\Phi_t$ is measurable, $\Phi^{\CDF}_t$ is also measurable for the same reason that $\Psi_t^{\CDF}$ is. Since $\zeta$ is stationary under the flow of $(\Phi_t)_{t\geq 0}$, the pushforward 
\[
\zeta^{\CDF}:=\kappa_{\#}\zeta\in \mathcal{P}(\calD)
\]
is stationary under $(\Phi_t^{\CDF})_{t\geq 0}$.

Solutions of \eqref{eq:hydrodynamic limit CDF} are everywhere continuous (by \cite[Theorem 1.1]{Berestycki2018}) and strictly decreasing in $x$ to the right of the free boundary at strictly positive times (by Remark \ref{rmk:U strictly decreasing at pve times}). Therefore given a sequence $(U_n)_{n=1}^{\infty}$ in $\calD$ converging to $U\in \calD$ pointwise at the continuity points of $U$, it follows from \eqref{eq:continuity of PsiCDF} that
\begin{equation}\label{eq:pointwise convergence of PhiCDF}
    \Phi^{\CDF}_t(U_n)(x)\ra \Phi^{\CDF}_t(U)(x)\quad\text{pointwise everywhere}
\end{equation}
for all $t>0$. In particular, for all $t>0$, $\Phi_t^{\CDF}:\calD\ra \calD$ (and hence $\Phi_t:\calP(\Rm)\ra \calP(\Rm)$) is continuous.

\medskip

Our main tool to finish the proof is a version of the so-called \emph{stretching lemma} which is a consequence of the maximum principle and which is one of the key tools for proving that solutions of the F-KPP equation converge to travelling waves.

\commentout{\OT{Remove}We will need to establish this stretching lemma for the following superset of $\calD$,
\[
\calS:=\{U:\Rm\ra [0,1]:\text{$U$ non-increasing, $\lim _{x\ra-\infty} U(x) =1$ and $\lim_{x\ra \infty} U(x) = 0$}\}.
\]
We note that $\calD$ consists precisely of those $U\in\calS$ which are c\`adl\`ag. The set $\calS$ will be needed in the proof of Lemma \ref{lem:convergence to minimal travelling wave from Heaviside}. }

Given $U,V\in \calD$, we say that $U$ is more stretched than $V$, denoted by $U\geq_s V$, if for any $c\in\Rm$ and $x_1\leq x_2$ we have
\begin{equation}\label{eq:defin 1 of being more stretched}
U(x_1)> V(x_1+c)\Rightarrow U(x_2)\geq V(x_2+c).
\end{equation}
Note that this differs slightly from the notion of being ``more stretched'' given for the F-KPP equation in Bramson's memoir \cite[p.33]{Bramson1983}. Since solutions of \eqref{eq:hydrodynamic limit CDF} are identically equal to $1$ to the left of the free boundary, the stretching lemma would not be true were we to define ``more stretched'' as in \cite[p.33]{Bramson1983}. Moreover, it will be important that the notion of being ``more stretched'' is preserved under pointwise limits, which is not the case with the definition given in \cite[p.33]{Bramson1983}. 
\begin{lem}\label{lem:stretchedness preserved under ptwise limits}
If $(U_n)_{n=1}^{\infty}$ and $(V_n)_{n=1}^{\infty}$ are two sequences in $\calD$ converging in $\tau_{\calD}$ to $U\in \calD$ and $V\in \calD$ respectively, and $U_n\geq_s V_n$ for all $n<\infty$, then $U\geq_s V$.
\end{lem}
\begin{proof}[Proof of Lemma \ref{lem:stretchedness preserved under ptwise limits}]
The case that $x_1=x_2$ is vacuous. We now suppose that $x_1< x_2$, $c\in \Rm$ and $U(x_1)>V(x_1+c)$. 

We firstly consider the case that $x_1$ and $x_2$ are continuity points of both $U$ and $V(\cdot+c)$, so that $U_n(x_i)\ra U(x_i)$ and $V_n(x_i+c)\ra V(x_i+c)$ as $n\ra\infty$, for $i=1,2$. Then $U_n(x_1)>V_n(x_1+c)$ eventually, hence $U_n(x_2)\geq V_n(x_2+c)$ eventually. Therefore $U(x_2)\geq V(x_2+c)$.

We now consider the case that either $x_1$ or $x_2$ is not a continuity point of either $U$ or $V(\cdot+c)$. There exists $\epsilon>0$ arbitrarily small such that $x_1+\epsilon<x_2$, $U(x_1+\epsilon)>V(x_1+c+\epsilon)$ and $x_1+\epsilon$ and $x_2+\epsilon$ are continuity points of both $U$ and $V(\cdot+c)$. Then the above implies that $U(x_2+\epsilon)\geq V(x_2+\epsilon+c)$. Now taking $\epsilon\downarrow 0$ and using the fact that $U$ and $V$ are c\`adl\`ag, we obtain that $U(x_2)\geq V(x_2+c)$.
\end{proof}

We define 
\begin{equation}\label{eq:ay definition}
a^y(U):=\inf\{x>-\infty:U(x)<y\}
\end{equation}
for $U\in \calD$ and $y\in (0,1]$, with $a^y(U):=-\infty$ if this set is empty. We observe that $a^1(U_t)\equiv L_t$ for any solution $(U,L)$ of \eqref{eq:hydrodynamic limit CDF}.

If $U,V\in \calD$ are continuous on $\{x:x\geq a^1(U)\}$ and $\{x:x\geq a^1(V)\}$ respectively (so in particular the only possible position for a discontinuity is at $a^1$), then $U \geq_s V $ is equivalent to saying that for all $y\in (0,1)$ we have
\begin{equation}\label{eq:defin of being more stretched}
\begin{aligned}
U(x+a^y(U))&\geq V(x+a^y(V) ),\quad x>0,\\
U(x+a^y(U))&\leq V(x+a^y(V)),\quad x<0.
\end{aligned}
\end{equation}
If in addition to $U \geq_s V $ we have $a^1(U),a^1(V)>-\infty$, then it follows that we have \eqref{eq:defin of being more stretched} for $y=1$. We note that this does allow for $U$ and $V$ to have downward jumps at $a^1(U)$ and $a^1(V)$ respectively, but we observe that the downward jump at $a^1(U)$ cannot be larger than that at $a^1(V)$.

We now state the version of the  stretching lemma we will need. 
\begin{lem}[Stretching lemma]\label{lem:extended maximum principle}
We suppose that $U_t$ and $V_t$ are solutions of \eqref{eq:hydrodynamic limit CDF} with $U_0,V_0\in \calD$ and $U_0\geq_s V_0$. Then $U_t\geq_s V_t$ for all $t\geq 0$. 
\end{lem}
Such a stretching lemma was first proven for the FKPP equation by Kolmogorov, Petrovskii and Piskunov \cite{Kolmogorov1937}.
\begin{proof}[Proof of Lemma \ref{lem:extended maximum principle}]
In the case of the FKPP with a general nonlinear term, the stretching lemma is given in Bramson's memoir \cite[Proposition 3.2, p.31]{Bramson1983}. Note that \cite[Proposition 3.2, p.31]{Bramson1983} provides for the stretching lemma where we define the notion of being more stretched as in \eqref{eq:defin 1 of being more stretched}. 

To obtain the stretching lemma for \eqref{eq:hydrodynamic limit CDF}, we fix arbitrary initial data $U_0,V_0\in \calD$ with $U_0\geq_s V_0$. We consider the FKPP with nonlinearity $U-U^n$, namely
\begin{equation}\label{eq:n-FKPP}
\partial_tU=\frac{1}{2}\Delta U+U-U^n.
\end{equation}
We denote by $U^{(n)}(x,t)$ and $V^{(n)}(x,t)$ the solutions of \eqref{eq:n-FKPP} with initial conditions $U_0$ and $V_0$ respectively. We may then pass to the limit as $n\ra\infty$ using \cite{Berestycki2018}, so that
\[
U^{(n)}(x,t)\ra U(x,t)\quad\text{and}\quad V^{(n)}(x,t)\ra V(x,t)
\]
pointwise as $n\ra\infty$. We fix $t>0$. We have that $U^{(n)}(\cdot,t)\geq_s V^{(n)}(\cdot,t)$ for all $n<\infty$ by \cite[Proposition 3.2, p.31]{Bramson1983}, hence $U(\cdot,t)\geq_s V(\cdot,t)$ by Lemma \ref{lem:stretchedness preserved under ptwise limits}.\end{proof}

We now prove the following lemma.
\begin{lem}[Boundary comparison lemma]\label{lem:boundary comparison lemma}
We suppose that $(U,L^U)$ and $(V,L^V)$ are solutions of \eqref{eq:hydrodynamic limit CDF} with $U_0,V_0\in \calD$, $L^U_0,L^V_0>-\infty$ and $U_0\geq_s V_0$. Then $L^U_t-L^U_0\geq L^V_t-L^V_0$ for all $t>0$.
\end{lem}
\begin{proof}[Proof of Lemma \ref{lem:boundary comparison lemma}]
We fix arbitrary $\epsilon>0$. Recentring if necessary, we may assume that $L^V_0=0$ and $L^U_0=2\epsilon$. Then $V_0(\epsilon)<1=U_0(\epsilon)$. Since $U\geq_s V$, it follows that $U_0(x)\geq V_0(x)$ for all $x>\epsilon$. Since $U_0(x)=1$ for $x\leq \epsilon$, $U_0(x)\geq V_0(x)$ everywhere. Then by the comparison principle \cite[Theorem 1.2]{Berestycki2018}, $U_t(x)\geq V_t(x)$ everywhere for all $t>0$. It follows that $U_t(x)=1$ for all $x<L^V_t$ and $t>0$, so that $L_t^U\geq L_t^V$ for all $t>0$. We have established that 
\[
L_t^U-L_0^U\geq L_t^V-L_0^V-2\epsilon
\]
for $t>0$. Since $\epsilon>0$ was arbitrary, we are done.
\end{proof}

\medskip
We have now gathered the necessary ingredients for the proof of Proposition \ref{prop:possible invariant measure unique}. 

We define $V_0:=\Ind(x<0)$ and $V_t:=\Psi_t(V_0)$. We write (recall $a^{\frac{1}{2}}$ was defined in \eqref{eq:ay definition})
\[
\tilde{\Pi}_{\min}(x):=\Pi_{\min}(a^{\frac{1}{2}}(\Pi_{\min})+x)\quad\text{and}\quad \tilde{V}_t:=V_t(a^{\frac{1}{2}}(V_t)+x)=\Phi_t^{\CDF}(V_0)(x),\quad x\in \Rm,\; t>0.
\]
We now employ the classical monotonicity argument of Kolmogorov, Petrovsky and Piskunov \cite{Kolmogorov1937} to prove the following:
\begin{lem}\label{lem:convergence to minimal travelling wave from Heaviside}
We have the convergence
\begin{equation}\label{eq:uniform convergence to travelling wave}
    \tilde{V}_t\ra \tilde{\Pi}_{\min}\quad \text{as $t\ra\infty$ uniformly in $x\in \Rm$.}
\end{equation}
\end{lem}
\begin{proof}[Proof of Lemma \ref{lem:convergence to minimal travelling wave from Heaviside}]
 Since $V_h\geq_s V_0$ for all $h\geq 0$, it follows from the stretching lemma (Lemma \ref{lem:extended maximum principle}) applied to initial data $V_h \geq_s V_0$  that $V_{t+h}\geq_s V_t\geq_s V_0$ for all $t,h\geq 0$. On the other hand, since $V_0\leq_s \Pi_{\min}$ and $\tilde{\Pi}_{\min}$ is fixed under $(\Phi_t^{\text{CDF}})_{t\geq 0}$, $V_t\leq_s \Pi_{\min}$ for all $t\geq 0$. It follows that 
\[
V_0\leq_s V_s\leq_s V_t\leq_s \Pi_{\min}
\]
for all $0\leq s\leq t<\infty$. It then follows from \eqref{eq:defin of being more stretched} that $\tilde{V}_t(x)$ is: 
\begin{enumerate}
    \item non-decreasing in $t$ with $\tilde{V}_t(x)\leq \tilde{\Pi}_{\min}(x)$ for $x> 0$;
    \item non-increasing in $t$ with $\tilde{V}_t(x)\geq \tilde{\Pi}_{\min}(x)$ for $x<0$;
    \item always equal to $\frac{1}{2}=\tilde{\Pi}_{\min}(0)$ for $x=0$ and $t>0$.
\end{enumerate}

It therefore follows that there exists a pointwise limit $\tilde{V}_t\ra \tilde{V}_{\infty}(x)$ as $t\ra \infty$, with $\tilde{V}_{\infty}(x)\leq \tilde{\Pi}_{\min}(x)$ for $x\geq 0$ and $\tilde{V}_{\infty}(x)\geq \tilde{\Pi}_{\min}(x)$ for $x\leq 0$. We now define $\tilde{V}_{\infty}^+$ to be the c\`adl\`ag modification of $\tilde{V}_{\infty}$, given by $\tilde{V}_{\infty}^+(x):=\lim_{h\downarrow 0}V_{\infty}(x+h)$. From the above, we see that $\tilde{V}_{\infty}^+\in \calD$. 

Since $\tilde{V}_t(x)\ra \tilde{V}_{\infty}^+(x)$ as $t\ra \infty$ at all continuity points of $\tilde{V}_{\infty}^+(x)$, $\tilde{V}_{t}\ra\tilde{V}_{\infty}^+$ in $\tau_{\calD}$ as $t\ra\infty$. Lemma \ref{lem:stretchedness preserved under ptwise limits} then implies that 
\[
\tilde{V}^+_{\infty}\leq \tilde{\Pi}_{\min}
\]
for all $t<\infty$. We obtain from \eqref{eq:pointwise convergence of PhiCDF} that
\[
\tilde{V}_{t+s}=\Phi^{\CDF}_s(\tilde{V}_t)\ra \Phi^{\CDF}_s(\tilde{V}^+_{\infty})
\]
in $\tau_{\calD}$ as $t\ra \infty$, hence $\Phi^{\CDF}_s(\tilde{V}_{\infty}^+)=\tilde{V}_{\infty}^+$ for all $s>0$. This implies that $\tilde{V}_{\infty}^+$ is a travelling wave for \eqref{eq:hydrodynamic limit CDF}. We take the solution $(V^{\infty},L^{\infty})$ of \eqref{eq:hydrodynamic limit CDF} with initial condition $\tilde{V}_{\infty}^+$. Since $\tilde{V}_{\infty}^+\leq_{s} \tilde{\Pi}_{\min}$, Lemma \ref{lem:boundary comparison lemma} implies that
\[
L^{\infty}_t\leq L^{\infty}_0+\sqrt{2}t
\]
for all $t\geq 0$, so $\tilde{V}_{\infty}^+$ is a travelling wave for \eqref{eq:hydrodynamic limit CDF} with velocity at most $\sqrt{2}$. The only possibility is that $\tilde{V}_{\infty}^+=\tilde{\Pi}_{\min}$. Finally, Dini's theorem yields that the convergence of $\tilde{V}_{t}$ to $\tilde{\Pi}_{\min}$ is uniform.
\end{proof}

The following lemma represents a crucial step in our proof.
\begin{lem}\label{lem:free boundary a.s. sqrt 2 t}
Let $U_0\sim \zeta^{\text{CDF}}$ be a random initial condition and consider $(U,L)$ the solution of \eqref{eq:hydrodynamic limit CDF} started from $U_0$ (so that $U(x,t)=\Psi_t^{\text{CDF}}(U_0)(x)$ almost surely). Then $L_t\equiv L_0+\sqrt{2}t$ almost surely.
\end{lem}
\begin{proof}[Proof of Lemma \ref{lem:free boundary a.s. sqrt 2 t}]
We consider arbitrary $U\in\mathcal{D}$ and keep the definition $V_0:=\Ind(x<0)$ and $V_t=\Psi^{\text{CDF}}_t(V_0)$. Since $U\geq_s V_0$, it follows from Lemma \ref{lem:extended maximum principle} that $\Phi_t^{\CDF}(U)\geq_s V_t$ for any $t<\infty$. 

We now take random $U_0\sim \zeta^{\text{CDF}}\in \mathcal{P}(\calD)$. It follows from the stationarity of $\zeta^{\CDF
}$ under $\Phi^{\CDF}_t$ that $U_0\overset{d}{=} \Phi_t(U_0)\geq_s V_0$ almost surely. Therefore $U_0\geq_s V_t$ for all $t<\infty$, almost surely. $U_0\geq_s V_t$ for all $t<\infty$ implies that $U_0\geq_s\tilde{\Pi}_{\min}$ by \eqref{eq:uniform convergence to travelling wave}. Therefore 
\begin{equation}\label{eq:almost surely more stretched than pimin}
U_0\geq_s \tilde{\Pi}_{\min}
\end{equation}
almost surely.

For (random) initial condition $U_0\in \calD$, $L_t=a^1(U_t)=a^1(\Psi_t^{\CDF}(U_0))$ is the corresponding free boundary at time $t$. Similarly we write $L^{\tilde{\Pi}_{\min}}_t$ for the free boundary with initial condition $\tilde{\Pi}_{\min}$, given by $L^{\tilde{\Pi}_{\min}}_t=L^{\tilde{\Pi}_{\min}}_0+\sqrt{2}t$. By Lemma \ref{lem:boundary comparison lemma}, if $U_0\geq_s\tilde{\Pi}_{\min}$ then
\[
L_t-L_0\geq L^{\tilde{\Pi}_{\min}}_t-L^{\tilde{\Pi}_{\min}}_0=\sqrt{2}t
\]
for all $t\geq 0$. 

Therefore having taken $U_0\sim \zeta^{\text{CDF}}$, it follows from \eqref{eq:almost surely more stretched than pimin} that
\[
L_t-L_0 \geq \sqrt{2}t
\]
for all $t\geq 0$, almost surely. On the other hand, since $\zeta$ is a subsequential limit of $\zeta^N$, it satisfies the conclusions of Proposition \ref{prop:characterisation of subsequential limits}, and we have that
\[
\expE_{U_0\sim \zeta^{\CDF
}}[L_t-L_0]\leq \sqrt{2}t.
\]
It therefore follows that $L_t-L_0\equiv \sqrt{2}t$, $\zeta^{\CDF
}$-almost surely.    
\end{proof}

We now return to the point of view given by \eqref{eq:hydrodynamic limit} (i.e. where total mass is $1$ and $u\equiv 0$ on $x\leq L_t$). We take $u_0\sim \zeta$. Since $L_t-L_0\equiv \sqrt{2}t$, it follows from the stochastic representation given in Appendix \ref{appendix:stochastic representation for PDE} that the evolution of $u_t$ is entirely equivalent to that of Brownian motion with constant drift $-\sqrt{2}$, killed instantaneously at $0$, which we denote by $(X_t:0\leq t<\tau_{0})$. To be more precise, defining $\hat{u}_t(dx):=u_t(x-L_t)dx$ (i.e. shifting so that the free boundary stays at $0$), it follows from Theorem \ref{theo:stochastic representation for PDE} that
\[
\hat{u}_t=\Law_{\hat{u}_0}(X_t\lvert \tau_{0}>t)\quad\text{for all}\quad t\geq 0\quad\text{and}\quad
\Law_{\hat{u}_0}(\tau_{0})=\text{exp}(1),
\]
for $\zeta$-almost every ${u}_0$. We have from \cite[Lemma 1.2]{Martinez1998} that $\Law_{\hat{u}_0}(\tau_{0})=\text{exp}(1)$ implies that $\hat{u}_0=\pi_{\min}$. We therefore conclude that $\zeta=\delta_{\tilde{\pi}_{\min}}$.
\end{proof}

We have established that $\zeta^N$ converges in $\calP(\calP(\Rm_{\geq 0}))$ to $\delta_{\tilde{\pi}_{\min}}$. In words, this means that the stationary distribution of the $N$-particle system centred by its median converges as $N\ra\infty$ to the minimal travelling wave centred by its median. Our next goal is to establish that this also implies that
\begin{equation}\label{eq:convergence centred by its leftmost final proof statement}
\chi^N\ra  \delta_{\pi_{\min}} 
\end{equation}
in $\mathcal{P}(\mathcal{P})(\Rm_{\geq 0})$ as $N\ra\infty$. In words, that the stationary distribution of the $N$-particle system centred by its leftmost particle converges as $N\ra\infty$ to the minimal travelling wave centred by its left boundary. To see why this is not immediate, observe that the former implies that a proportion $1-o(1)$ of the particles (the ``bulk'' of particles) are arranged according to the profile of the minimal travelling wave. This does not preclude, however, the possibility that a small $o(1)$ proportion of particles are to the left of the bulk. This would imply that once we centre by the leftmost particle, the bulk of particles will be to the right of $0$, hence resemble the minimal travelling wave with its leftmost tip to the right of $0$. 
\begin{proof}[Proof of \eqref{eq:convergence centred by its leftmost final proof statement}]
We firstly recall that Proposition \ref{prop:chiN tight} ensured that $\{\chi^N:N<\infty\}$ is tight in $\calP(\calP(\Rm_{\geq 0}))$. We take a subsequential limit of $\chi^N$, which we denote by $\chi$. We have by \eqref{eq:expected value of difference between median and minima} that $\{\Law_{\mu\sim \chi^N}(A(\mu)):N<\infty\}$ is tight in $\calP(\Rm_{\geq 0})$, so we may take a further subsequence along which it converges in distribution. Then by Skorokhod's representation theorem we have on some new probability space $(\tilde{\Omega},\tilde{\mathcal{F}},\tilde{\Pm})$ (we write $\tilde{\expE}$ for expectation with respect to $\tilde{\Pm}$) along this subsequence random measures $\mu_N\sim \chi^N$, $\mu\sim \chi$ and a random variable $\tilde{A}$ such that
\[
\mu_N\ra \mu \quad\text{$\tilde{\Pm}$-almost surely in $\Wah$ and}\, A(\mu_N)\ra \tilde{A} \;\text{$\tilde{\Pm}$-almost surely as $N\ra\infty$.}
\]
Then we have that $\mu(dx)=\tilde{\pi}_{\min}(x-\tilde{A})dx$ $\tilde{\Pm}$-almost surely, so that for some random variable $\tilde{H}$ to be determined we have $\mu(dx)=\pi(x-\tilde{H})dx$ $\tilde{\Pm}$-almost surely. We have that $\tilde{H}\geq 0$ since $\mu_N$ is supported on $\Rm_{\geq 0}$ ($\tilde{\Pm}$-almost surely) for all $N<\infty$. We recall that we established in Proposition \ref{prop:chiN tight} that
\[
\tilde{\expE}\Big[\int_{\Rm_{\geq 0}}x\mu_N(dx)\Big]=v_N\leq \sqrt{2}\quad\text{for all $N<\infty$.}
\]
We calculate from \eqref{eq:minimal travelling wave} that $\int_0^{\infty}x\pi_{\min}(dx)=\sqrt{2}$, hence
\[
\int_0^{\infty}x\mu(dx)=\tilde{H}+\sqrt{2}
\]
$\tilde{\Pm}$-almost surely. It follows from Fatou's lemma that 
\begin{equation}\label{eq:equation for showing tilde H = 0}
\begin{aligned}
\tilde{\expE}[\tilde{H}]+\sqrt{2}&=\tilde{\expE}\Big[\int_{\Rm_{\geq 0}}x\mu (dx)\Big]\leq \tilde{\expE}\Big[\liminf_{N\ra\infty}\int_{\Rm_{\geq 0}}x\mu _N(dx)\Big] \overset{\text{Fatou}}\leq  \liminf_{N\ra\infty}\tilde{\expE}\Big[\int_{\Rm_{\geq 0}}x\mu _N(dx)\Big]\leq \sqrt{2}.
\end{aligned}
\end{equation}
from which we see that $\tilde{H}=0$ $\tilde{\Pm}$-almost surely. It follows that $\mu_N\ra \pi_{\min}$ in $\Wah$ $\tilde{\Pm}$-almost surely as $N\ra\infty$. We conclude \eqref{eq:convergence centred by its leftmost final proof statement}.
\end{proof}
The only thing left is to strengthen the notion of convergence to 
\begin{equation}\label{eq:final claimed stronger notion of convergence}
\expE_{\vec{Y}^N\sim \psi^N}[\Wah_1(\Theta^N(\vec{Y}^N),\pi_{\min})]\ra 0\quad\text{as}\quad N\ra\infty.
\end{equation}
We define, as in \eqref{eq:mean measures}, the mean measures
\[
\xi^N(\cdot):=\tilde{\expE}[\mu_N(\cdot)].
\]
Since we must have equality in \eqref{eq:equation for showing tilde H = 0}, it follows that
\begin{equation}\label{eq:convergence of first moment}
\int_{\Rm_{\geq 0}}x\xi_N (dx)\ra\sqrt{2}=\int_{\Rm_{\geq 0}}x\pi_{\min}(dx)\quad\text{as}\quad N\ra\infty.
\end{equation}
We now take $g_n\in C(\Rm_{\geq 0};[0,1])$ such that $g_n\equiv 1$ on $[0,n]$ and $g_n\equiv 0$ on $[n+1,\infty)$, for $n\geq 1$, and further define $h_n:=1-g_n$. Then for any $n<\infty$ we have
\begin{equation}\label{eq:convergence with gn term}
\int_{\Rm_{\geq 0}}xg_n(x)\xi_N(dx)\ra \int_{\Rm_{\geq 0}}xg_n(x)\pi_{\min}(dx)\quad\text{as}\quad N\ra\infty.
\end{equation}
It follows, by considering the difference of \eqref{eq:convergence of first moment} and \eqref{eq:convergence with gn term}, that for all $n<\infty$ we have
\[
\int_{\Rm_{\geq 0}}xh_n(x)\xi_N(dx)\ra \int_{\Rm_{\geq 0}}xh_n(x)\pi_{\min}(dx)\quad\text{as}\quad N\ra\infty.
\]
We conclude that for any $\epsilon>0$ there exists $C_{\epsilon}<\infty$ such that
\[
\tilde{\expE} \Big[\int_{[C_{\epsilon},\infty)}x\mu_N(dx)\Big]\leq \epsilon\quad\text{for all}\quad N<\infty.
\]
We therefore conclude \eqref{eq:final claimed stronger notion of convergence}. This completes the proof of Theorem \ref{theo:main theorem}.
\qed 

\begin{appendix}

\section{Proof of Theorem \ref{theo:hydrodynamic limit for random ics}}\label{appendix:full appendix with hydro and sensitivity}

We begin by recalling some notation. We recall that $\calP_f$ was defined in \eqref{eq:calPfdefin} by 
\[
\calP_f(\Rm):=\{\mu\in \calP(\Rm):L(\mu)>-\infty\},\quad\text{where}\quad L(\mu):=\inf\{x:\mu([x,\infty))<1\}.
\]
We further recall that given $\mu\in\calP_f(\Rm)$ we write $\Psi_t(\mu)$ for the solution at time $t$ of \eqref{eq:hydrodynamic limit} with $u_0 = \mu$. This then defines a flow on $\calP_f(\Rm)$ (and $\calP(\Rm)$), denoted by $(\Psi_t)_{t\geq 0}$. 

The continuity (and hence measurability) of $\Psi_t:\calP(\Rm)\ra \calP(\Rm)$ will be proven in Appendix \ref{appendix:continuity of flow} using \eqref{eq:distance under coupling exponential bound}, found in Appendix \ref{appendix:hydrodynamic limit fixed ic}. We need to establish the measurability of $\Psi_t:\calP_f(\Rm)\ra \calP_f(\Rm)$ in order for the statement of Theorem \ref{theo:hydrodynamic limit for random ics} to make sense. This is because, given a random variable $\mu\sim \Lambda$ as in the statement of Theorem \ref{theo:hydrodynamic limit for random ics}, we need $\Psi_t$ to be measurable in order for $\Psi_t(\mu)$ to be a random variable. Without this, the law $\Law_{\mu\sim \Lambda}(\mu,\Psi_t(\mu))$ is meaningless.

We will therefore proceed as follows:
\begin{enumerate}
    \item In Appendix \ref{appendix:hydrodynamic limit fixed ic} we will establish Theorem \ref{theo:hydrodynamic limit conv in prob}, a hydrodynamic limit theorem in which we assume that the initial conditions converge in probability to a \textit{deterministic} initial condition. Neither the statement of this theorem nor its proof will require knowing the measurability of $\Psi_t:\calP(\Rm)\ra \calP(\Rm)$.
    \item In Appendix \ref{appendix:continuity of flow} we shall establish the continuity (hence measurability) of $\Psi_t:\calP(\Rm)\ra\calP(\Rm)$, obtaining an explicit estimate in terms of the $\Wah$ metric.
    \item Finally in Appendix \ref{appendix:hydrodynamic limit random ic}, we shall conclude the proof of Theorem \ref{theo:hydrodynamic limit for random ics}.
\end{enumerate}

We now define some more notation. We write 
\[
d_1(x,y):=\lvert x-y\rvert\wedge 1,\quad\text{for} \quad x,y\in \Rm. 
\]
The metric $\Wah$ defined in Definition \ref{defin:Wasserstein} is then the Wasserstein metric generated by $d_1$. For $\mu_1,\mu_2\in\calP(\Rm)$, we write $\mu_1\leq \mu_2$ if $\mu_2$ stochastically dominates $\mu_1$. Finally, for $\mu\in\calP(\Rm)$ and $c\in \Rm$ we write $\mu+c$ for $\mu$ shifted by $c$, that is 
\[
\mu+c:=(x\mapsto x+c)_{\#}\mu.
\]

\begin{theo}\label{theo:hydrodynamic limit conv in prob}
We suppose that we have a sequence of $N$-BBMs $(\vec{X}^N_t)_{t\geq 0}$ with initial conditions such that $m^N_0$ converges weakly in probability to some given $\mu\in\calP_f(\Rm)$. Then $m^N_t$ converges weakly in probability to $\Psi_t(\mu)$, for any given $t>0$.
\end{theo}

\subsection{Proof of Theorem \ref{theo:hydrodynamic limit conv in prob}}\label{appendix:hydrodynamic limit fixed ic}
De Masi et al. \cite{DeMasi2019} imposed two assumptions on the initial conditions of the $N$-BBMs which we would like to remove: 
\begin{enumerate}
\item that $\mu$ has an absolutely continuous density;
\item that at time $0$ and for all $N$, $X^{1}_0,\ldots,X^{N}$ are independent and identically distributed.
\end{enumerate}

We will proceed in two steps, removing the first assumption in the first step, then the second assumption in the second step.

\underline{Step 1}

We proceed by a sandwiching argument. We fix arbitrary $\epsilon>0$. We take $\phi_{-}\in C_c^{\infty}((-\epsilon,0);\Rm_{\geq 0})$ with $\int_{\Rm}(\phi_{-})(x)dx=1$, and $\phi_{+}(x):=\phi_{-}(x-\epsilon)\in C_c^{\infty}((0,\epsilon))$. Then by convolution we obtain 
\[
\mu_{-}:=\phi_{-}\ast \mu \leq \mu\leq \mu_{+}:=\phi_{+}\ast \mu.
\]
We observe that $\mu_{+},\mu_{-}$ have absolutely continuous densities with 
\[
L(\mu_{-}),L(\mu_{+})>-\infty\quad \text{and}\quad \mu_{+}=\mu_{-}+\epsilon.
\]

We define $\vec{X}^N_0=(X^1_0,\ldots,X^N_0)$ to be such that $X^{1}_0,\ldots,X^N_0$ are independent and identically distributed according to $\mu$. We then take $\delta_{i}$ for $1\leq i\leq N$ to be independent random variables with distribution $\phi_{-}(x)dx$, and define $X^{i,-}_0:=X^{i}_0+\delta_0$. We set $X^{i,+}_0:=X^{i,-}_0+\epsilon$. We define $\vec{X}^{N,-}_0:=(X^{1,-}_0,\ldots,X^{N,-}_0)$ and $\vec{X}^{N,+}_0:=(X^{1,+}_0,\ldots,X^{N,+}_0)$, observing that the sequences $X^{1,-}_0,\ldots,X^{N,-}_0$ and $X^{1,+}_0,\ldots,X^{N,+}_0$ are each independent and identically distributed according to $\mu_{-}$ and $\mu_+$ respectively (although the two sequences are obviously not independent of each other).

We now permute the indices of $\vec{X}^{N,-}_0$, $\vec{X}^{N}_0$ and $\vec{X}^{N,+}_0$ respectively so that $X^1_0\leq X^2_0\leq\ldots\leq X^N_0$, and similarly for the other two. We observe that with this labelling of the indices,
\[
X^{i,-}_0\leq X^i_0\leq X^{i,+}_0\quad\text{and}\quad X^{i,+}_0=X^{i,-}_0+\epsilon
\]
for all $i\in \{1,\ldots,N\}$.

We now take $N$-BBMs $\vec{X}^{N,-}_t$, $\vec{X}^N_t$ and $\vec{X}^{N,+}_t$ with initial conditions given by $\vec{X}^{N,-}_0$, $\vec{X}^{N}_0$ and $\vec{X}^{N,+}_0$ respectively, coupled as follows. In between killing times we define all three particle systems to be given by normally reflected diffusions in $\{(x_1,\ldots,x_N)\in \Rm^N:x_1\leq \ldots \leq x_N\}$, driven by the same $N$-dimensional Brownian motion (which is possible since \cite{Lions1984} established the existence of strong solutions for the corresponding SDE). The killing events are catalysed by the same rate-$N$ Poisson clock, at which time we choose a rank $k\in \{1,\ldots,N\}$ uniformly at random and declare that the $k^{\text{th}}$ ranked particle in each particle system branches. Indices are then permuted at this time to preserve the ordering. We see that
\begin{equation}\label{eq:inequalities between particles appendix pf}
X^{i,-}_t\leq X^i_t\leq X^{i,+}_t\quad\text{and}\quad X^{i,+}_t=X^{i,-}_t+\epsilon
\end{equation}
for all $i\in \{1,\ldots,N\}$ and $t\geq 0$. 

The empirical measures of the three $N$-BBMs are denoted by $m^{N,-}_t:=\Theta^N_{\#}\vec{X}^{N,-}_t$, $m^N_t:=\Theta^N_{\#}\vec{X}^N_t$ and $m^{N,+}_t:=\Theta^N_{\#}\vec{X}^{N,+}_t$. By applying \cite[Theorem 1]{DeMasi2019} to $\vec{X}^{N,-}$, we see that 
\[
\expE[\Wah(m^{N,-}_t,\Psi_t(\mu_{-}))]\ra 0
\]
as $N\ra\infty$. On the other hand, it follows from \eqref{eq:inequalities between particles appendix pf} that
\[
\expE[\Wah(m^{N}_t,m^{N,-}_t)]\leq \expE[\Wah(m^{N,+}_t,m^{N,-}_t)]\leq \epsilon.
\]
Moreover, since $\mu_-\leq \mu\leq \mu_+=\mu_-+\epsilon$, it follows from the comparison principle \cite[Theorem 1.1]{Berestycki2018} that $\Psi_t(\mu_-)\leq \Psi_t(\mu)\leq \Psi_t(\mu_+)=\Psi_t(\mu_-)+\epsilon$, so that $\Wah(\Psi_t(\mu_-),\Psi_t(\mu))\leq \epsilon$. 
It follows that 
\begin{align*}
&\limsup_{N\ra\infty}\expE[\Wah(m^N_t,\Psi_t(\mu))]\\&\leq \limsup_{N\ra\infty}\expE[\Wah(m^N_t,m^{N,-}_t)]+\limsup_{N\ra\infty}\expE[\Wah(m^{N,-}_t,\Psi_t(\mu_-))]+\Wah(\Psi_t(\mu_-),\Psi_t(\mu))\leq 2\epsilon.   
\end{align*}

Using that $\epsilon>0$ is arbitrary, we have removed the assumption that $\mu$ has a density.

\underline{Step 2}

We now fix $\mu\in \calP_f(\Rm)$. We take two sequences of $N$-BBMs. The first, denoted by $(\vec{X}^N_t)_{t\geq 0}$ and with empirical measure $m^N_t:=\Theta^N_{\#}\vec{X}^N_t$, is such that $m^N_0\ra \mu$ weakly in probability as $N\ra\infty$. The second, denoted by $\tilde{\vec{X}}^N_t$ and with empirical measure $\tilde{m}^N_t:=\Theta^N_{\#}\tilde{\vec{X}}^N_t$, is such that $\tilde{X}^1_0,\ldots,\tilde{X}^N_0$ are independent and identically distributed with distribution $\mu$. Our goal is to construct a coupling of these two $N$-BBMs such that for all $N<\infty$ we have
\begin{equation}\label{eq:distance under coupling exponential bound}
\expE[\Wah(m^N_t,\tilde{m}^N_t)]\leq e^t\expE[\Wah(m^N_0,\tilde{m}^N_0)]\quad \text{for all}\quad 0\leq t<\infty.
\end{equation}

The following shall employ basic concepts from optimal transport which can be found in \cite{Villani2009}. Our strategy will be to use the solution of the Monge problem at successive times to construct a coupling between the $N$-BBMs. We shall establish that the optimal solution of the Monge problem gives an optimal solution of the Kantorovich problem (by which $\Wah$ is defined), so that the cost of the optimal solution of the Monge problem is exactly the $\Wah$ distance between the particle systems. This will allow us to control the growth of $\Wah(m^N_t,\tilde{m}^N_t)$.

Given two sequences in $\Rm$ of the same length, $(x_1,\ldots,x_n)$ and $(y_1,\ldots,y_n)$, it is trivial that there exists a bijection $\iota:\{1,\ldots,N\}\ra \{1,\ldots,N\}$ minimising $\frac{1}{n}\sum_{i=1}^nd_1(x_i,y_{\iota(i)})$ - the solution of the Monge problem. We fix such an $\iota$. Then $\{(x_i,y_{\iota(i)}):1\leq i\leq n\}$ is \textit{$c$-cyclically monotone}, the definition of this being given in \cite[Definition 5.1]{Villani2009}. It then follows from \cite[Theorem 1.3]{Pratelli2008} that $\frac{1}{n}\sum_{i=1}^n\delta_{(x_i,y_{\iota(i)})}$ (with the projections $(x,y)\mapsto x$ and $(x,y)\mapsto y$) is an optimal coupling for the Kantorovich problem for $d_1$. That is,
\[
\Wah(\frac{1}{N}\sum_{i=1}^N\delta_{x_i},\frac{1}{N}\sum_{i=1}^N\delta_{y_i})=\frac{1}{N}\sum_{i=1}^Nd_1(x_i,y_{\iota(i)}).
\]
We call such a coupling a ``Monge-optimal coupling''. We note it may not be unique. 

For the coupled $N$-BBMs $\vec{X}^N_t$ and $\tilde{\vec{X}}^N_t$ to be constructed, a ``Monge-optimal coupling'' at time $t$, denoted by $\iota_t:\{1,\ldots,N\}\ra \{1,\ldots,N\}$, is a bijection such that 
\[
\Wah(m^N_t,\tilde{m}^N_t)=\frac{1}{N}\sum_{i=1}^Nd_1(X^i_t,\tilde{X}^{\iota_t(i)}_t).
\]

We now construct our desired coupling. We begin by defining a sequence of $\text{exp}(N)$ killing times $0<\tau_1<\tau_2<\ldots$, which will catalyse branching-selection events for both processes. We further define $\tau_0:=0$. Throughout, whenever we take a Monge-optimal coupling, it should be understood that we take one uniformly at random if it's not unique. At time $0$, we take a Monge-optimal coupling $\iota_0$. We then drive $X^i_t$ and $\tilde{X}^{\iota_0(i)}_t$ by the same Brownian motion $W^i_t$ up to time $\tau_1$. We see that 
\[
\Wah(m^N_s,\tilde{m}^N_s)=\frac{1}{N}\sum_{i=1}^Nd_1(X^i_s,\tilde{X}^{\iota_s(i)}_s)\leq \frac{1}{N}\sum_{i=1}^Nd_1(X^i_s,\tilde{X}^{\iota_0(i)}_s)=\frac{1}{N}\sum_{i=1}^Nd_1(X^i_0,\tilde{X}^{\iota_0(i)}_0)= \Wah(m^N_0,\tilde{m}^N_0)
\]
for all $0\leq s<\tau_1$.

We now let $i_{\ast}\in \{1,\ldots,N\}$ and $j_{\ast}\in \{1,\ldots,N\}$ be the index of the minimal particle in $\vec{X}^N_{\tau_1-}$ and $\tilde{\vec{X}}^N_{\tau_1-}$ respectively. We must now determine which particle in each particle systems branches. 

With probability $\frac{1}{N}$, we select the minimal particle for both $N$-BBMs - that is the particle $X^{i_{\ast}}_{\tau_1-}$ and $X^{j_{\ast}}_{\tau_1-}$ respectively, in which case $\Wah(m_{\tau_1},\tilde{m}_{\tau_1})=\Wah(m_{\tau_1-},\tilde{m}_{\tau_1-})$ (as the two particle systems haven't changed). 

Otherwise we take a Monge-optimal coupling $\iota':\{1,\ldots,N\}\setminus \{i_{\ast}\}\ra \{1,\ldots,N\}\setminus \{j_{\ast}\}$. We claim that 
\begin{equation}\label{eq:Monge optimal restriction}
\frac{1}{N}\sum_{i\neq i_{\ast}}d_1( X^{i}_{\tau_1-},\tilde{X}^{\iota'(i)}_{\tau_1-})\leq \Wah (m_{\tau_1-},\tilde{m}_{\tau_1-}).
\end{equation}
To see this, we take the Monge-optimal coupling $\iota_{\tau_1-}:\{1,\ldots,N\}\ra\{1,\ldots,N\}$. If $\iota_{\tau_1-}(i_{\ast})=j_{\ast}$, then by restriction we have
\begin{align*}
&\frac{1}{N}\sum_{i\neq i_{\ast}}d_1( X^{i}_{\tau_1-}-\tilde{X}^{\iota'(i)}_{\tau_1-})\leq \frac{1}{N}\sum_{i\neq i_{\ast}}d_1( X^{i}_{\tau_1-}-\tilde{X}^{\iota_{\tau_1-}(i)}_{\tau_1-})\\
&\leq \frac{1}{N}\sum_{i}d_1( X^{i}_{\tau_1-}-\tilde{X}^{\iota_{\tau_1-}(i)}_{\tau_1-})=\Wah (m_{\tau_1-},\tilde{m}_{\tau_1-}),
\end{align*}
so we have established \eqref{eq:Monge optimal restriction} in this case. Otherwise, there exists $i'\in \{1,\ldots,N\}\setminus \{i_{\ast}\}$ and $j'\in\{1,\ldots,N\}\setminus \{j_{\ast}\}$ such that $\iota_{\tau_1-}(i')=j_{\ast}$ and $\iota_{\tau_1-}(i_{\ast})=j'$. We define the coupling
\[
\begin{aligned}
\iota'':\{1,\ldots,N\}\setminus \{i_{\ast}\}&\ra \{1,\ldots,N\}\setminus \{j_{\ast}\},\\
i&\mapsto \begin{cases}
j',\quad i=i'\\
\iota_{\tau_1-}(i),\quad \text{otherwise.}
\end{cases}
\end{aligned}
\]
We observe that $d_1( X^{i'}_{\tau_1-},\tilde{X}^{j'}_{\tau_1-})\leq d_1( X^{i'}_{\tau_1-},\tilde{X}^{j_{\ast}}_{\tau_1-})+d_1( X^{i_{\ast}}_{\tau_1-},\tilde{X}^{j'}_{\tau_1-})$, which may be seen by considering separately the possibilities $X^{i'}_{\tau_1-}\leq X^{j'}_{\tau_1-}$ (in which case the left hand side is at most the second term on the right) and $X^{i'}_{\tau_1-}\geq X^{j'}_{\tau_1-}$  (in which case the left hand side is at most the first term on the right). From this it follows that
\[
\begin{aligned}
&\frac{1}{N}\sum_{i\neq i_{\ast}}d_1(X^{i}_{\tau_1-},\tilde{X}^{\iota''(i)}_{\tau_1-})= \frac{1}{N}\sum_{i\neq i_{\ast},i'}d_1( X^{i}_{\tau_1-},\tilde{X}^{\iota_{\tau_1-}(i)}_{\tau_1-})+\frac{1}{N}d_1( X^{i'}_{\tau_1-},\tilde{X}^{j'}_{\tau_1-})
\\ &\leq \frac{1}{N}\sum_{i\neq i_{\ast},i'}d_1( X^{i}_{\tau_1-},\tilde{X}^{\iota_{\tau_1-}(i)}_{\tau_1-}) +\frac{1}{N}d_1( X^{i'}_{\tau_1-},\tilde{X}^{j_{\ast}}_{\tau_1-})+\frac{1}{N}d_1( X^{i_{\ast}}_{\tau_1-},\tilde{X}^{j'}_{\tau_1-})=\Wah(m_{\tau_1-},\tilde{m}_{\tau_1-}).
\end{aligned}
\]
We then obtain \eqref{eq:Monge optimal restriction} from the optimality of $\iota'$. 

We now choose $i\in \{1,\ldots,N\}\setminus \{i_{\ast}\}$ uniformly at random, declare that $X^{i_{\ast}}_{\tau_1-}$ jumps onto $X^i_{\tau_1-}$, and declare that $\tilde{X}^{j_{\ast}}_{\tau_1-}$ jumps onto $\tilde{X}^{\iota'(i_{\ast})}_{\tau_1-}$. We see that the expected value of $d_1(X^{i_{\ast}}_{\tau_1},X^{j_{\ast}}_{\tau_1})$ after doing this is at most $\frac{N}{N-1}\Wah(m_{\tau_1-},\tilde{m}_{\tau_1-})$. 

Overall, we see that the expected value of $\Wah(m_{\tau_1},\tilde{m}_{\tau_1})$, conditional on the $N$-BBMs at time $\tau_1-$, is at most 
\[
\frac{1}{N}\Wah(m_{\tau_{1}-},\tilde{m}_{\tau_1-})+ \frac{N-1}{N}\Big[\Wah(m_{\tau_1-},\tilde{m}_{\tau_1-})+\frac{1}{N-1}\Wah(m_{\tau_1-},\tilde{m}_{\tau_1-})\Big]=(1+\frac{1}{N})\Wah(m_{\tau_1-},\tilde{m}_{\tau_1-}).
\]
In the above terms on the left-hand side, the first term corresponds to the possibility of jumping onto the minimal particle (so that the particle systems don't change), whilst the second and third terms (in the square brackets) correspond to the event whereby this does not happen. Of these, the former is the upper bound \eqref{eq:Monge optimal restriction} for $\frac{1}{N}\sum_{i\neq i_{\ast}}d_1(X^{i}_{\tau_1},X^{\iota'(i)}_{\tau_1})$ (corresponding to particles which haven't moved during the jump), while the latter is an upper bound for the expected value of $\frac{1}{N}d_1(X^{i_{\ast}}_{\tau_1},X^{j_{\ast}}_{\tau_1})$ immediately after the jump.

We repeat this coupling inductively, obtaining $(\vec{X}^N_t)_{t\geq 0}$ and $(\tilde{\vec{X}}^N_t)_{t\geq 0}$ for all time. We now fix $t>0$ and define the discrete-time process
\[
M_k:=\Wah(m_{\tau_k\wedge t},\tilde{m}_{\tau_k\wedge t})-\Wah(m_{0},\tilde{m}_{0})-\frac{1}{N}\sum_{\substack{0< r\leq  k\\\text{such that}\\0< \tau_{r}\leq t}}\Wah(m_{\tau_r-},\tilde{m}_{\tau_r-}).
\]
Writing $(\mathcal{F}_t)_{t\geq 0}$ for the filtration on which we have defined the above coupled $N$-BBMs, we see from the above that $M_k$ is a discrete-time $(\mathcal{F}_{\tau_k\wedge t})_{k\geq 0}$-supermartingale. Moreover, since $\Wah\leq 1$ and the number of killing events in time $t$ is $\text{Poi}(Nt)$, $\lvert M_k\rvert$ is stochastically dominated by $1+\frac{1}{N}\text{Poi}(Nt)$ for all $k$, hence is uniformly integrable. It follows that
\[
\expE[\Wah(m_{ t},\tilde{m}_{ t})]\leq \expE[\Wah(m_{0},\tilde{m}_{0})]+\frac{1}{N}\expE[\sum_{0<\tau_k\leq t}\Wah(m_{\tau_k-},\tilde{m}_{\tau_k-})]\quad \text{for all}\quad 0\leq t<\infty.
\]
On the other hand, as in \eqref{eq:mg 3 tightness pf} we have that
\[
\frac{1}{N}\sum_{0< \tau_k\leq t}\Wah(m_{\tau_k-},\tilde{m}_{\tau_k-})-\int_0^t\Wah(m_{s},\tilde{m}_s)ds\quad\text{is an $(\mathcal{F}_t)_{t\geq 0}$ martingale.}
\]
We see that
\[
\expE[\Wah(m_t,\tilde{m}_t)]\leq \expE[\Wah(m_0,\tilde{m}_0)]+\int_0^t\expE[\Wah(m_s,\tilde{m}_s)]ds\quad\text{for all $0\leq t<\infty$.}
\]
We therefore obtain \eqref{eq:distance under coupling exponential bound} by applying Gronwall's inequality.

We now apply \cite[Theorem 1]{DeMasi2019}, with the assumption that $\mu$ is absolutely continuous removed (using Step 1), to see that $\tilde{m}^N_t\ra \Psi_t(\mu)$ in $\Wah$ in probability as $N\ra\infty$. Since $\expE[\Wah(m^N_0,\tilde{m}^N_0)]\ra 0$ as $N\ra\infty$, it follows from \eqref{eq:distance under coupling exponential bound} that $\expE[\Wah(m^N_t,\tilde{m}^N_t)]\ra 0$ as $N\ra\infty$. It therefore follows that $m^N_t\ra \Psi_t(\mu)$ in $\Wah$ in probability.
\qed


\subsection{Sensitivity of solutions of \eqref{eq:hydrodynamic limit} to the initial condition}\label{appendix:continuity of flow}
\begin{lem}\label{lem:sensitivity depending on ics}
Let $u(x,t)$ and $v(x,t)$ be two solutions of \eqref{eq:hydrodynamic limit} with initial conditions $u_0,v_0\in \calP(\Rm)$. Then we have that
\begin{equation}\label{eq:sensitivity on ics for free boundary PDE}
    \Wah(u_t,v_t)\leq e^{t}\Wah(u_0,v_0)
\end{equation}
for all $t>0$. In particular, $\Psi_t:\calP(\Rm)\ra \calP(\Rm)$ is continuous (hence measurable).
\end{lem}
\begin{proof}[Proof of Theorem \ref{lem:sensitivity depending on ics}]
We firstly assume that $u_0,v_0\in \calP_f(\Rm)$. We take a sequence of $N$-BBMs $(\vec{X}^N_t)_{t\geq 0}$ and $(\tilde{\vec{X}}^N_t)_{t\geq 0}$ with initial conditions given by $\vec{X}^N_0\sim u_0^{\otimes N}$, $\tilde{\vec{X}}^N_0\sim v_0^{\otimes N}$. We write $m_t^N=\Theta^N(\vec{X}^N_t)$, $\tilde{m}^N_t=\Theta^N(\tilde{\vec{X}}^N_t)$ for the corresponding empirical measures. Then by the proof of \eqref{eq:distance under coupling exponential bound} (the assumptions being made here on the initial conditions are slightly different but no changes need to be made to the proof) we have
\[
\expE[\Wah(m^N_t,\tilde{m}^N_t)]\leq e^t\expE[\Wah(m^N_0,\tilde{m}^N_0)].
\]
Taking $\limsup_{N\ra\infty}$ of both sides and applying Theorem \ref{theo:hydrodynamic limit conv in prob}, we obtain \eqref{eq:sensitivity on ics for free boundary PDE} in the limit.

We now remove the assumption that $u_0,v_0\in \calP_f(\Rm)$, assuming only that $u_0,v_0\in \calP(\Rm)$. We have that $u_{\epsilon},v_{\epsilon}\in \calP_f(\Rm)$ for all $\epsilon>0$ by \cite{Berestycki2018}, hence 
\begin{equation}\label{eq:Wasserstein estimate start pve times}
\Wah(u_t,v_t)\leq e^{t-\epsilon}\Wah(u_{\epsilon},v_{\epsilon})
\end{equation}
for all $\epsilon>0$. Since $u_{\epsilon}\ra u_0$ and $v_{\epsilon}\ra v_0$ weakly (hence in $\Wah$) as $\epsilon\ra 0$, from the definition of solutions to \eqref{eq:hydrodynamic limit}, we can take the limit of \eqref{eq:Wasserstein estimate start pve times} as $\epsilon\ra 0$ with fixed $t>0$ to obtain \eqref{eq:sensitivity on ics for free boundary PDE}
\end{proof}

\subsection{Conclusion of the proof of Theorem \ref{theo:hydrodynamic limit for random ics}}\label{appendix:hydrodynamic limit random ic}
Having proven Theorem \ref{theo:hydrodynamic limit conv in prob} and Lemma \ref{lem:sensitivity depending on ics}, we are now in a position to prove Theorem \ref{theo:hydrodynamic limit for random ics}. We recall that whereas Theorem \ref{theo:hydrodynamic limit conv in prob} assumes that the initial conditions converge weakly in probability to a deterministic measure, Theorem \ref{theo:hydrodynamic limit for random ics} allows for the initial conditions to converge to a random limit. 

We now take a sequence of $N$-BBMs $(\vec{X}^N_t)_{t\geq 0}$ with empirical measures $m_t^N:=\Theta^N(\vec{X}^N_t)$ for $t\geq 0$, such that $\Law(m^N_0)\in \calP(\calP(\Rm))$ converges to $\Lambda\in \calP(\calP(\Rm))$ as $N\ra\infty$, with the assumption that $\Lambda(\calP_f(\Rm))=1$. We fix $t>0$. Our goal is to show that $\Law((m_0^N,m_t^N))$ converges in $\calP(\calP(\Rm)\times \calP(\Rm))$ to $\Law_{\mu\sim\Lambda}(\mu,\Psi_t(\mu))$. 

We metrise $\calP(\Rm)\times \calP(\Rm) $ using 
\[
d_{\calP(\Rm)\times \calP(\Rm) }((\mu_1,\nu_1),(\mu_2,\nu_2 )):= \Wah(\mu_1,\mu_2)+\Wah(\nu_1,\nu_2).
\]
We further take the Wasserstein-$1$ metric on $\calP(\calP(\Rm) \times \calP(\Rm))$ generated by $d_{\calP(\Rm)\times \calP(\Rm) }\wedge 1$ (see Definition \ref{defin:Wasserstein}), which we denote by $\overline{\Wah}$, and which metrises the topology of weak convergence of measures on $\calP(\calP(\Rm)\times\calP(\Rm))$. 

We define $\calP_N(\Rm)$ to be the space of probability measures on $\Rm$ of the form $\frac{1}{N}\sum_{i=1}^N\delta_{x^i}$.

\commentout{
We firstly establish the following lemma. \JB{We can change $G_t^N$ and make $\Xi$ acts directly in $\mu_N$}
\begin{lem}\label{lem:measurable mapping from N-measures to N-vectors}
There exists, for all $N\geq 2$, a measurable function
\[
\iota_N=(\iota_{N,1},\ldots,\iota_{N,N}):\calP_N(\Rm)\ra \Rm^N
\]
such that $\frac{1}{N}\sum_{i=1}^N\delta_{\iota_{N,i}(\mu)}=\mu$ for all $\mu\in\calP_N(\Rm)$.
\end{lem}
\begin{proof}[Proof of Lemma \ref{lem:measurable mapping from N-measures to N-vectors}]
Inductively define $\iota_{N,1}(\mu):=\inf\{x:x\in \text{spt}(\mu)\}$, $\iota_{N,k}(\mu):=\inf\{x:x\in \text{spt}(\mu-\frac{1}{N}\sum_{r<k}\delta_{\iota_{N,r}(\mu)})\}$ for $2\leq k\leq N$.
\end{proof}}

Using Skorokhod's representation theorem, we define on a common probability space $\Omega^{\text{ic}}$ with probability measure $\Pm^{\text{ic}}$ a sequence of $\calP_N(\Rm)$-valued random measures $\mu_N$ for $N<\infty$, and a $\calP_f(\Rm)$-valued random measure $\mu$, such that $\Law(\mu_N)=\Law(m^N_0)$ for all $N<\infty$, $\Law(\mu)=\Lambda$, and $\mu_N\ra \mu$ $\mathbb{P}^{\text{ic
}}$-almost surely in $\Wah$ as $N\ra\infty$. We further define on a separate filtered probability space $\tilde \Omega$ with probability measure $\tilde{\Pm}$ the necessary Brownian motions and Poisson point processes encoding the movement and branching of the particles of the $N$-BBM, for all $N<\infty$. In words, the probability space $\Omega^{\text{ic}}$ determines our initial condition, whilst the second probability space determines the evolution thereafter. 

Observe that any measure $\mu \in \calP_N(\Rm)$ can be viewed as the initial configuration of the $N$-BBM, in the sense that we can always associate a random vector $\vec{X}_N(\mu)$ in the following way: inductively define $X_1(\mu):=\inf\{x:x\in \text{spt}(\mu)\}$, $X_{k}(\mu):=\inf\{x:x\in \text{spt}(\mu-\frac{1}{N}\sum_{r<k}\delta_{X_{r}(\mu)})\}$ for $2\leq k\leq N$. 

Thus, on $\Omega^{\text{ic}}\times \tilde \Omega$, we have a sequence of $N$-BBMs, whose initial configurations are given by the sequence $(\mu_N)$ in the aforedescribed manner, with the evolution driven by $\tilde \Omega$. These are equal in distribution to the original $N$-BBMs, $\vec X^N_t$, up to relabelling of the indices.



For each $N<\infty$ and $t\geq 0$, we take the measurable function 
\[
G^N_t:\calP_N(\Rm)\times \tilde{\Omega}\ra \calP_N(\Rm)
\]
such that $G^N_t(\mu,\tilde{\omega})$ is the $N$-BBM at time $t$ with initial condition $\mu$, and driving Brownian motions and jumps given by $\tilde{\omega}\in \tilde{\Omega}$. We then define 
\[
\Xi^N_t(\mu,\tilde{\omega}):=(\mu, G^N_t(\mu,\tilde{\omega}))\in \mathcal{P}(\Rm)\times \mathcal{P}(\Rm).
\]
This gives the empirical measures of the initial condition in the first coordinate and of the $N$-BBM at time $t$ in the second coordinate. 

We see that 
\[
\Law(m^N_0,m^N_t)=\Law^{\Pm^{\text{ic}}\times \tilde{\Pm}}(\Xi^N_t(\mu_N),\tilde{\omega}).
\]
We have from Theorem \ref{theo:hydrodynamic limit conv in prob} that for all $\omega^{\text{ic}}\in \Omega^{\text{ic}}$ such that $\mu_N(\omega^{\text{ic}})\ra \mu(\omega^{\text{ic}})\in \calP_f(\Rm)$, we have
\[
\expE^{\tilde{\Pm}} \Big[d_{\calP(\Rm)\times \calP(\Rm) }\big[\Xi^N_t(\mu_N(\omega^{\text{ic}}),\tilde{\omega}),\{\mu(\omega^{\text{ic}}),\Phi_t(\mu(\omega^{\text{ic}})) \}\big]\Big]\ra 0\quad \text{as}\quad N\ra\infty.
\]
It follows from the bounded convergence theorem and Fubini's theorem that
\begin{equation}\label{eq:convergence of Expectation Pic times tilde P}
\expE^{\Pm^{\text{ic}}\times\tilde{\Pm}}
\Big[d_{\calP(\Rm)\times \calP(\Rm) }\big[\Xi^N_t(\mu_N(\omega^{\text{ic}}),\tilde{\omega}),\{\mu(\omega^{\text{ic}}),\Phi_t(\mu(\omega^{\text{ic}})) \}\big] \wedge 1 \Big]
\ra 0\quad \text{as}\quad N\ra\infty.
\end{equation}

For each $N\ge 2,$ we have constructed a coupling of $\Law(m^N_0,m^N_t )$ and $\Law_{\mu\sim \Lambda}(\mu,\Phi_t(\mu) )$ with the cost in the Kantorovich problem for $d_{\mathcal{P}(\Rm)\times \mathcal{P}(\Rm)}\wedge 1$ for these couplings being by definition the left-hand side of \eqref{eq:convergence of Expectation Pic times tilde P}. This, of course, bounds the cost of the optimal coupling (which defines $\overline{\Wah}$). Therefore \eqref{eq:convergence of Expectation Pic times tilde P} implies that 
\[
\overline{\Wah}(\Law(m^N_0,m^N_t ),\Law_{\mu\sim \Lambda}(\mu,\Phi_t(\mu) )) \ra 0
\]
as $N\ra\infty$. This completes the proof of Theorem \ref{theo:hydrodynamic limit for random ics}.
\qed

\section{Stochastic representation for \eqref{eq:hydrodynamic limit}}\label{appendix:stochastic representation for PDE}
We provide here a stochastic representation for solutions of \eqref{eq:hydrodynamic limit}. Whilst this was certainly previously known, the authors are not aware of a proof having been written down in the literature. 

We denote Lebesgue measure on $\Rm$ as $\Leb(\cdot)$. For any $\mu\in \calP(\Rm)$, Brownian motion $(B_t)_{0\leq t<\infty}$ with initial condition $\mu$ will be assumed to be defined on a filtered probability space we denote by $(\Omega,\mathcal{F},(\mathcal{F}_t)_{t\geq 0},\Pm_{\mu})$, with $\Law_{\mu}(B_0)=\mu$. 


\begin{theo}\label{theo:stochastic representation for PDE}
Let $(u,L)$ be a solution to \eqref{eq:hydrodynamic limit} with initial condition $u_0\in \calP(\Rm)$. We then take $B_t$ to be a Brownian motion with initial condition $B_0\sim u_0$, killed instantaneously at the time $\tau:=\inf\{t>0:B_t\leq L_t\}$. We have $\Pm_{u_0}(\tau>t)>0$ for all $t<\infty$, so that the conditional law $\Law_{u_0}(B_t\lvert \tau>t)(\cdot)$ is well-defined. We have that
\begin{equation}\label{eq:PDE solution given by density of BM}
    \Law_{u_0}(B_t\lvert \tau>t)(\cdot)=u_t(\cdot)d \Leb(\cdot)
\end{equation}
for all $t>0$. Moreover we have
\begin{equation}\label{eq:killing time for free boundary exp(1)}
    \Law_{u_0}(\tau)=\text{exp}(1).
\end{equation}
\end{theo}
\begin{rmk}
It may be the case that $B_0=L_0$ with positive probability. For instance if we take $u_0=\delta_0$, then $B_0=L_0=0$ almost surely. It is important, therefore, that in the definition of the stopping time $\tau$ we only consider strictly positive times $t$ such that $B_t\leq L_t$. For the same Dirac initial condition, if $L_t$ were to be Lipschitz at $t=0$, then we would have $\tau=0$ almost surely. We observe, however, that \eqref{eq:killing time for free boundary exp(1)} includes the statement that $L_t$ always moves to the left sufficiently quickly so that $\tau>0$ almost surely.
\end{rmk}
Before proving Theorem \ref{theo:stochastic representation for PDE}, we must firstly establish the following lemma.
\begin{lem}\label{lem:Lipschitz from the left free bdy}
For any $T>0$ there exists $C_{T}>-\infty$ such that $L_{t_2}-L_{t_1}\geq C_{T}(t_2-t_1)$ for all $t_2\geq t_1\geq T$ and for any solution $(u,L)$ of \eqref{eq:hydrodynamic limit}.
\end{lem} 
\begin{rmk}\label{rmk:remark on two types of hitting time equality due to Lipschitz}
It follows from Lemma \ref{lem:Lipschitz from the left free bdy} that $\inf\{t>0:B_t\leq L_t\}=\inf\{t>0:B_t<L_t\}$, where $(B_t)_{0\leq t<\infty}$ is a Brownian motion with initial condition $u_0$ and $(u,L)$ is a solution of \eqref{eq:hydrodynamic limit} with initial condition $u_0$, for any $u_0\in \calP(\Rm)$.
\end{rmk}
\begin{proof}[Proof of Lemma \ref{lem:Lipschitz from the left free bdy}]
By the correspondence between solutions $(u,L)$ of \eqref{eq:hydrodynamic limit} and solutions $(U,L)$ of \eqref{eq:hydrodynamic limit CDF}, it suffices to establish the statement of Lemma \ref{lem:Lipschitz from the left free bdy} with solutions of \eqref{eq:hydrodynamic limit} replaced by solutions of \eqref{eq:hydrodynamic limit CDF}.

We will make use of Lemmas \ref{lem:extended maximum principle} and \ref{lem:boundary comparison lemma}, the proofs of which do not make use of any results from this Appendix. We write $U\geq_s V$ to mean that $U$ is more stretched than $V$, as defined in \eqref{eq:defin 1 of being more stretched}.

We fix $T>0$ and let $(V,L^V)$ be the solution of \eqref{eq:hydrodynamic limit CDF} with Heaviside initial condition $V_0(\cdot)=\Ind(\cdot<0)$. Then $V_t\geq_s V_0$ for any $t>0$, hence $V_t\geq_s V_s$ for all $0\leq s\leq t$, by Lemma \ref{lem:extended maximum principle}. Lemma \ref{lem:boundary comparison lemma} therefore implies that 
\begin{equation}\label{eq:monotonicity of boundary}
    \Rm_{\geq 0}\ni t\mapsto L^V_{t+h}-L^V_t\in \Rm\quad\text{is a non-decreasing function for any $h>0$.}
\end{equation} 

We define 
\[
f:[0,T]\ni t\mapsto L^V_t-L^V_0-(L^V_T-L^V_0)\frac{t}{T}\in \Rm.
\]
We observe that $f$ is a continuous function, $f(0)=f(T)=0$ and $f(\frac{T}{2})\leq 0$. The first of these follows since \cite[Theorem 1.1 and Proposition 1.3]{Berestycki2018} implies that $[0,T]\ni t\mapsto L^V_t\in \Rm$ is continuous, the second is obvious, and the third follows since since $L^V_T-L^V_{\frac{T}{2}}\geq L^V_{\frac{T}{2}}-L^V_0$, by \eqref{eq:monotonicity of boundary}.

We may therefore fix $T'\in (0,T)$ such that $f(T')\leq f(t)$ for all $t\in [0,T]$. Then 
\[
L^V_t-L^V_{T'}=f(t)-f(T')+(L^V_T-L^V_0)\frac{t-T'}{T}\geq \frac{L^V_T-L^V_0}{T}(t-T')
\]
for all $t\in [T',T]$. It then follows from \eqref{eq:monotonicity of boundary} that $L^V_s-L^V_t\geq C_T(s-t)$ for all $T\leq t\leq s<\infty$, where 
\[
C_T:=\frac{L^V_T-L^V_0}{T}>-\infty
\]

Finally, since $U_0\geq V_0$ for any initial condition $U_0$, Lemmas \ref{lem:extended maximum principle} and \ref{lem:boundary comparison lemma} imply that $L_s-L_t\geq L^V_s-L^V_t$ for any solution $(U,L)$ of \eqref{eq:hydrodynamic limit CDF}, whence Lemma \ref{lem:Lipschitz from the left free bdy} follows.
\end{proof}

\begin{proof}[Proof of Theorem \eqref{theo:stochastic representation for PDE}]
In the following, a classical solution, $v$, of the heat equation on an open set is defined to be a $C^{2,1}$ function (which is necessarily $C^{\infty,\infty}$ on the same open set)  satisfying $\partial_tv=\frac{1}{2}\Delta v$ pointwise. Moreover $(\Omega,\mathcal{F},(\mathcal{F}_t)_{t\geq 0},(B_t)_{t\geq 0},(\Pm_x)_{x\in \Rm})$ is $1$-dimensional Brownian motion with $\Pm_x(B_0=x)=1$ for all $x\in \Rm$. We define $\Pm_{\mu}:=\int_{\Rm}\Pm_x(\cdot)\mu(dx)$, so that $(B_t)_{t\geq 0}$ is Brownian motion with initial condition $B_0\sim \mu$ under $\Pm_{\mu}$. We write $\expE_x$ and $\expE_{\mu}$ for expectation with respect to $\Pm_x$ and $\Pm_{\mu}$ respectively. 

 We firstly prove the following elementary lemma, a version of the parabolic maximum principle. 
\begin{lem}\label{lem:maximum principle}
Suppose that $v(x,t)$ is a classical solution of the heat equation $\partial_tv=\frac{1}{2}\Delta v$ on the open set $\calO \subseteq\Rm\times (0,T)$, for some time horizon $T<\infty$. We write $v_t(x):=v(x,t)$. We assume that (1) $v$ is bounded on  $(\Rm\times (\epsilon,T))\cap \calO$, for all $\epsilon>0$. We define $\Sigma_0:=\partial \calO\cap (\Rm\times \{0\})$ and $\Sigma_L:=\partial \calO\cap (\Rm\times (0,T))$. We further assume that (2) $\limsup_{(x',t')\ra (x,t)}v(x',t')\leq 0$ for all $(x,t)\in \Sigma_L$, and (3) $v_t(\cdot)d\Leb(\cdot)\ra 0$ in the topology of weak convergence of measures as $t\downarrow 0$. Then $v(x,t)\leq 0$ on $\calO$.
\end{lem}
\begin{proof}[Proof of Lemma \ref{lem:maximum principle}]
We assume for the time being that $v$ is non-negative.

In the following, Brownian motion will have initial condition $B_0=0$ almost surely. We fix $(x,t)\in \calO$ and take the stopping time $\tau_0:=\inf\{s>0:(x-B_s,t-s)\notin \calO\}\leq t$. We further define the stopping times $\tau_{\delta}:=\inf\{s>0:d((x-B_s,t-s),\calO^c)\leq \delta\}$ for $\delta>0$, observing that $\tau_{\delta}\uparrow \tau_0$ almost surely as $\delta\downarrow 0$.

 We fix $\epsilon>0$.  It follows from Ito's lemma, from the boundedness of $v$ and from the optional stopping theorem that 
\[
v(x,t)=\expE_0[v_{\epsilon}(B_{t-\epsilon})\Ind(\tau_{\delta}>t-\epsilon)]+\expE_0[v_{\tau_{\delta}}(B_{t-\tau_{\delta}})\Ind(\tau_{\delta}\leq t-\epsilon)].
\]
By Assumption (1), we can apply Fatou's lemma to see that
\[
\limsup_{\delta\ra 0}\expE_0[v_{\tau_{\delta}}(B_{t-\tau_{\delta}})\Ind(\tau_{\delta}\leq t-\epsilon)]\leq \expE_0[\limsup_{\delta \ra 0}v_{\tau_{\delta}}(B_{t-\tau_{\delta}})\Ind(\tau_{\delta}\leq t-\epsilon)]\leq 0,
\]
the last inequality following from Assumption (2). Again using Fatou's lemma, it follows that
\begin{align*}
&v(x,t)\leq \limsup_{\delta\ra 0}\expE_0[v_{\epsilon}(B_{t-\epsilon})\Ind(\tau_{\delta}>t-\epsilon)]\\
&\leq \expE_0[\limsup_{\delta\ra 0}v_{\epsilon}(B_{t-\epsilon})\Ind(\tau_{\delta}>t-\epsilon)]\leq \expE_0[v_{\epsilon}(B_{t-\epsilon})\Ind(\tau_0>t-\epsilon)].
\end{align*}
We now write $k_t(x,y)$ for the Gaussian kernel. We see that
\[
v(x,t+\epsilon)\leq \int_{\Rm}k_{t}(x,y)v_{\epsilon}(y)dy\ra 0
\]
as $\epsilon\ra 0$ by Assumption (3) (using that $y\mapsto k_t(x,y)$ is continuous and bounded). Thus $v(x,t)\leq 0$ by the continuity of $v$. This concludes the proof of Lemma \ref{lem:maximum principle} under the additional assumption that $v$ is non-negative.

We now remove the assumption that $v$ is non-negative. We assume for contradiction that $\calO':=\{(x,t)\in \calO:v(x,t)>0\}$ is non-empty (otherwise we are done). Then since $v$ is continuous, $\calO'$ is an open subset of $\calO$ (and hence of $(0,T)\times \Rm$), and moreover Assumptions (1)-(3) remain true if we replace $\calO$ with $\calO'$. Having proven \ref{lem:maximum principle} when $v$ is non-negative, it follows that $v\leq 0$ on $\calO'$, which is a contradiction. Therefore $\calO'=\emptyset$ and we are done.
\end{proof}

We now let $(u,L)$ be a solution to \eqref{eq:hydrodynamic limit} with initial condition $u_0\in \calP(\Rm)$. We firstly prove Theorem \ref{theo:stochastic representation for PDE} with the additional assumption that $u_0$ is atomless, so impose this additional assumption for the time being.

We have from \cite[Theorem 1.1, Proposition 1.3 and Corollary 2.1]{Berestycki2018} that $ t\mapsto L_t\in \Rm$ is continuous on $(0,\infty)$, and that $L_t\ra L_0:=\inf\{x \in \Rm:x\in \text{spt}(u_0)\}\in \Rm\cup\{-\infty\}$ as $t\downarrow 0$. Since $u_0$ is atomless, $u_0((L_0,\infty))=1$. It therefore follows immediately from the parabolic Harnack inequality that $\Pm_{u_0}(\tau>t)>0$ for all $t<\infty$. Moreover continuity of $L_t$ and of the sample paths of Brownian motion imply that $\tau>0$ almost surely.


We now define for each $R\in \Nm$ a curve $L_{R}\in C^{\infty}(\Rm_{>0};\Rm)\cap C(\Rm_{\geq 0};\Rm\cup\{-\infty\})$ such that $L_{R}(t)\in (L(t),L(t)+\frac{1}{R})$ for $t>0$. We thereby define the time-dependent domains 
\[
\calO_R:=\{(x,t)\in \Rm\times \Rm_{>0}:L_{R}(t)<x\},\; R\in \Nm,\quad \calO:=\{(x,t)\in \Rm\times \Rm_{>0}:L(t)<x\}.
\]
\begin{center}
  \hspace{0.5cm}  \includegraphics[width=0.7\textwidth]{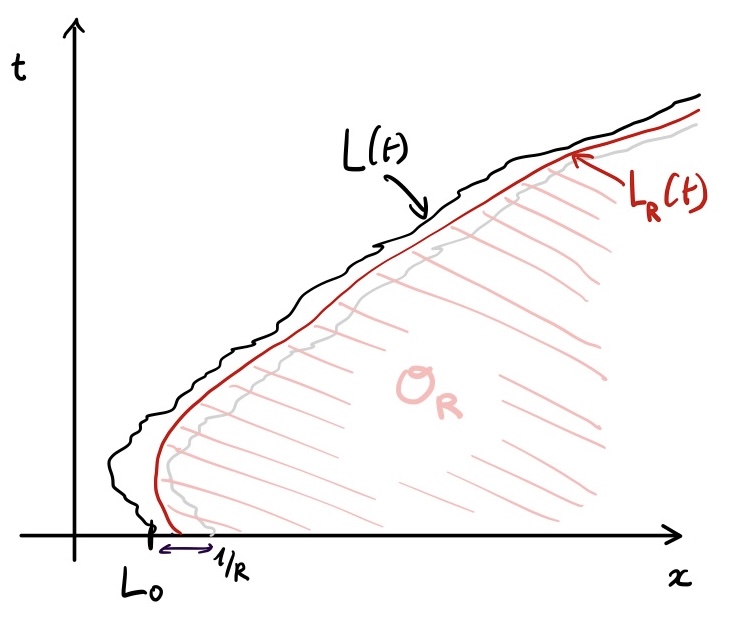}   
\end{center}

We write $v(x,t):=e^{-t}u(x,t)$. We see that $v$ is a classical solution of the heat equation on $\calO$. We also have from \cite[Theorem 1.1 and Corollary 2.1]{Berestycki2018} that 
\begin{equation}\label{eq:v satisfies Dirichlet bdy conditions}
v\in C(\Rm\times (0,T))\quad\text{with} \quad v\equiv  0\quad\text{on}\quad \{(L(t),t):t\in (0,T)\},
\end{equation}
i.e. $v$ satisfies Dirichlet boundary conditions along $L$.

We consider Brownian motion started from initial condition $u_0$, killed either at the times $\tau_R:=\inf\{t>0:(B_t,t)\notin \calO_R\}$ or at the time $\tau:=\inf\{t>0:(B_t,t)\notin \calO\}$. Then $\Pm_{u_0}(B_t\in \cdot,\tau>t)$ and $\Pm_{u_0}(B_t\in \cdot,\tau_R>t)$ have densities with respect to Lebesgue measure on $\calO$ and $\calO_R$ respectively, 
versions of which are $C^{\infty}$ classical solutions of the heat equation on $\calO$ and $\calO_R$ respectively. This is well-known to hold under the much more general parabolic H\"ormander conditions, by the argument given at the beginning of the proof of \cite[Theorem 3]{Ichihara1974} and H\"ormander's theorem \cite{Hormander1967}). This is a fact about diffusions on the interior of open sets, which doesn't require any information about boundary regularity. We write $\tilde{v}(x,t)$ and $\tilde{v}^R(x,t)$ respectively for these densities,  with $\tilde{v}_t(x):=\tilde v(x,t)$ and $\tilde{v}^R_t(x):=\tilde{v}^R(x,t)$. 

Our goal is to show the stochastic representation
\[
v\equiv \tilde{v}\quad\text{on}\quad \calO.
\]
We would like to appeal to suitable PDE uniqueness theory, but must be careful because: the boundary is only known to be continuous, it's not clear in which sense $\tilde{v}$ vanishes along $L(t)$, the domain is unbounded, and $v_t$ is only known to converge to $v_0$ weakly in measure as $t\downarrow 0$. The first three problems are not present on the domain $\calO_R$ for $R<\infty$ (but the final one is),  hence why we have introduced these domains. The final problem is dealt with by applying the form of the maximum principle we have established in Lemma \ref{lem:maximum principle}.


We claim that both $\tilde{v}^R-v$ and $v-\tilde{v}$ satisfy the assumptions of Lemma \ref{lem:maximum principle}, on the domains $\calO_R$ and $\calO$ respectively. It is immediate that they both satisfy Assumption (1).

Using the time-reversibility of Brownian motion, we have that
\[
\begin{split}
&\tilde{v}^R(x,t)dx=\int_{\Rm}\tilde{v}^R(y,\epsilon)\Pm_{y}[B_{t-\epsilon}\in dx,(B_s,\epsilon+s)\in \calO_R\text{ for all }0\leq s\leq t-\epsilon]dy\\
&=dx\int_{\Rm}\tilde{v}^R(y,\epsilon)\Pm_{x}[B_{t-\epsilon}\in dy,(B_s,t-s)\in \calO_R\text{ for all }0\leq s\leq t-\epsilon].
\end{split}
\]
Since $x\mapsto \tilde{v}^R(x,t)$ is continuous,
\[
\tilde{v}^R(x,t)=\expE_x[\tilde{v}^R(B_{t-\epsilon},\epsilon)\Ind((B_s,t-s)\in \calO_R\text{ for all }0\leq s\leq t-\epsilon)].
\]
Since $L^R\in C^{\infty}(\Rm_{>0};\Rm)$, it follows that $\tilde{v}^R$ vanishes continuously along $\partial \calO_R\cap (\Rm\times (0,T))$, which implies that $\tilde{v}^R-v$ satisfies Assumption (2). It similarly follows from \eqref{eq:v satisfies Dirichlet bdy conditions} that $v-\tilde{v}$ satisfies Assumption (2).

We have from \cite[Theorem 1.1 and Corollary 2.1]{Berestycki2018} that 
\[
v_t(\cdot)d\Leb(\cdot)\ra u_0(\cdot)
\]
weakly as $t\ra\infty$. We define 
\[
v^R_t(x):=v_t(x)\Ind((x,t)\in \calO_R)\quad\text{ and}\quad u^R_0(dx):=\Ind[x\in (L_R(0),\infty)]u_0(dx). 
\]
Since $L^R\in C^{\infty}(\Rm_{>0};\Rm)\cap C(\Rm_{\geq 0};\Rm\cup\{-\infty\})$ and $u_0$ is atomless, 
\[
v^R_t(\cdot)d\Leb(\cdot)\ra u^R_0(\cdot)d\Leb(\cdot)
\]
weakly as $t\ra \infty$. Moreover since $u_0$ is atomless and Brownian motion has continuous sample paths, we that 
\[
\tilde{v}^R_t(\cdot)d\Leb(\cdot)\ra u^R_0(\cdot)\quad\text{and}\quad\tilde{v}_t(\cdot)d\Leb \ra u_0(\cdot)
\]
weakly as $t\ra 0$. Putting these together, we see that 
\[
[\tilde{v}^R_t(\cdot)-v^R_t(\cdot)]d\Leb(\cdot)\ra 0\quad\text{and}\quad [v_t(\cdot)-\tilde{v}_t(\cdot)]d\Leb(\cdot)\ra 0
\]
weakly as $t\ra 0$. Therefore $\tilde{v}^R-v$ and $v-\tilde{v}$ satisfy Assumption (3) on the domains $\calO_R$ and $\calO$ respectively.

We can therefore apply Lemma \ref{lem:maximum principle} to both $\tilde{v}^R-v$ and $v-\tilde{v}$ to see that
\[
(\tilde{v}^R-v)(x,t)\leq 0\quad\text{and}\quad (v-\tilde{v})(x,t)\leq 0
\]
for all $(x,t)\in \calO_R$ and for all $(x,t)\in \calO$ respectively. 

We note that for all $(x,t)\in \calO$ we have that $(x,t)\in \calO_R$ for all $R$ sufficiently large. Using the monotone convergence theorem and the fact $u_0$ is atomless we see that 
\[
\int_{A}\tilde{v}^R_t(x)\Leb(dx)=\Pm_{u_0}(B_t\in A,\tau_R>t)\uparrow \Pm_{u_0}(B_t\in A,\tau>t)=\int_{A}\tilde{v}_t(x)\Leb(dx)
\]
as $R\ra\infty$, for all Borel-measurable $A\subseteq \Rm$ and $t>0$. Therefore, $\tilde{v}^R(x,t)\rightarrow \tilde{v}(x,t)$ for almost every $(x,t)\in \calO$ as $R\ra\infty$. Thus $\tilde{v}\leq v$ almost everywhere on $\calO$. Since both $\tilde{v}$ and $v$ are continuous on $\calO$, $\tilde{v}\leq v$ everywhere. Combined with the fact that $v\leq \tilde{v}$ everywhere, we see that $v\equiv \tilde{v}$ on $\calO$.

We have established that $u_t(\cdot)d\Leb(\cdot)=e^t\Pm_{u_0}(B_t\in \cdot,\tau>t)$ for all $t>0$. Integrating both sides, we see that $\Pm_{u_0}(\tau>t)=e^{-t}$ for all $t>0$, from which we conclude \eqref{eq:PDE solution given by density of BM} and \eqref{eq:killing time for free boundary exp(1)}, under the additional assumption that $u_0$ is atomless.

We now remove the assumption that $u_0$ is atomless. The key additional difficulty is to show that $\tau>0$ almost surely. For $\mu_1,\mu_2\in\calP(\Rm)$, we write $\mu_1\leq \mu_2$ if $\mu_2$ stochastically dominates $\mu_1$. We fix $\phi_{n}\in C_c^{\infty}((2^{-n-1},2^{-n});\Rm_{\geq 0})$ with $\int_{\Rm}\phi_{n}(x)dx=1$ for all $n<\infty$. We then define
\[
u^{n}_0:=\phi_{n}\ast u_0
\]
for $n<\infty$. We observe that 
\begin{equation}\label{eq:stochastic ordering of initial conditions mollification final appendix}
u_0 \leq  u^{n+1}_0\leq u^n_0
\end{equation}
for $n<\infty$. Moreover, we have that $u^{n}_0$ has an absolutely continuous density for all $n<\infty$, and that $u^{n}_0\ra u_0$ weakly as $n\ra \infty$.

We recall that given $\mu\in\calP(\Rm)$ we write $\Psi_t(\mu)$ for the solution at time $t$ of \eqref{eq:hydrodynamic limit} with $u_0 = \mu$. 

We let $(u^n,L^n)$ be the solution to \eqref{eq:hydrodynamic limit} with initial condition $u^n_0$. By \eqref{eq:stochastic ordering of initial conditions mollification final appendix} and the comparison principle \cite[Theorem 1.2]{Berestycki2018}, $L_t\leq \liminf_{n\ra\infty}L^n_t$ for all $t>0$, and moreover $L^n_t$ is pointwise non-increasing in $n$. On the other hand, Lemma \ref{lem:sensitivity depending on ics} and the upper semicontinuity of the map $L$ defined in \eqref{eq:lower bdy and median maps} (see \eqref{eq:A and L upper semicontinuous}) imply that 
\[
\limsup_{n\ra\infty}L^n_t=\limsup_{n\ra\infty}L(\Psi_t(u^n_0))\leq L(\Psi_t(u_0))=L_t.
\]
It therefore follows that $L^n_t\ra L_t$ pointwise as $n\ra\infty$. Since $L_t$ is continuous and $L^n_t$ is pointwise non-increasing in $n$, Dini's theorem implies that $L^n_t\ra L_t$ as $n\ra\infty$ uniformly on compact subsets of $\Rm_{>0}$.

Using Skorokhod's representation theorem, we take on $(\tilde{\Omega},\tilde{\mathcal{F}},(\tilde{\mathcal{F}}_t)_{t\geq 0},\tilde{\Pm})$ a family of $\tilde{\mathcal{F}}_0$-measurable random variables $(\tilde{X}_n)_{n=1}^{\infty}$, another $\tilde{\mathcal{F}}_0$-measurable random variable $\tilde{X}$, and an independent $(\tilde{\mathcal{F}}_t)_{t\geq 0}$-Brownian motion $(\tilde{B}_t)_{t\geq 0}$ with $\tilde{X}_n\sim u_0^n$ for $n<\infty$, $X\sim u_0$, $\tilde{X}_n\ra \tilde{X}$ $\tilde{\Pm}$-almost surely as $n\ra \infty$, and $\tilde{B}_0=0$ $\tilde{\Pm}$-almost surely. We write $\tilde{\expE}$ for expectation under $\tilde{\Pm}$. 

We now fix $0<\delta'<\delta<\infty$. It follows from the above and Fatou's lemma that
\[
\begin{split}
&\tilde{\Pm}(\tilde{B}_t+\tilde{X}<L_t\text{ for some $t\in [\delta',\delta]$})\leq \tilde{\expE}[\liminf_{n\ra\infty}\Ind(\tilde{B}_t+\tilde{X}_n<L^n_t\text{ for some $t\in [\delta',\delta]$})]\\
&\leq \liminf_{n\ra\infty}\tilde{\Pm}(\tilde{B}_t+\tilde{X}_n<L^n_t\text{ for some $t\in [\delta',\delta]$})\leq \delta.
\end{split}
\]
The last inequality follows since $\tilde B_t +\tilde X_n $ is a Brownian motion started from $u_0^n$, which is atomless. Hence, the hitting time of the free boundary $L^n_t$ is exponential by the first part of the proof.

Therefore, using Lemma \ref{lem:Lipschitz from the left free bdy} and Remark \ref{rmk:remark on two types of hitting time equality due to Lipschitz}, 
\[
\Pm_{u_0}(\tau< \delta)=\Pm_{u_0}(B_t<L_t\text{ for some $t\in (0,\delta)$})=\tilde{\Pm}(\tilde{B}_t+\tilde{X}<L_t\text{ for some }t\in (0,\delta))<\delta.
\]

It follows that, $\Pm_{u_0}$-almost surely, $\tau>0$. The fact that $\Pm_{u_0}(\tau>t)>0$ for any $t<\infty$ follows, as before, from the parabolic Harnack inequality.

The continuity of the sample paths of Brownian motion yields that 
\begin{equation}\label{eq:continuity at 0 of Brownian motion conditioned not to hit free bdy}
    \Law_{u_0}(B_h\lvert \tau>h)\ra u_0
\end{equation}
weakly as $h\ra 0$. Moreover the parabolic Harnack inequality implies that $\Law_{u_0}(B_h\lvert \tau>h)$ is atomless (in fact it has a bounded density with respect to Lebesgue measure) for any $h>0$. We can therefore apply Theorem \ref{theo:stochastic representation for PDE} for atomless initial conditions to $\Law_{u_0}(B_h\lvert \tau>h)$ for any $h>0$. We obtain that
\begin{equation}\label{eq:composition to obtain stochastic rep}
\Law_{u_0}(B_{h+t}\lvert \tau>h+t)=\Psi_t(\Law_{u_0}(B_{h}\lvert \tau>h))
\end{equation}
for any $h,t>0$. Using \eqref{eq:continuity at 0 of Brownian motion conditioned not to hit free bdy} and Lemma \ref{lem:sensitivity depending on ics}, we obtain \eqref{eq:PDE solution given by density of BM} by fixing $t\in \Rm_{>0}$ and taking the limit of \eqref{eq:composition to obtain stochastic rep} as $h\ra 0$.

Since we have established \eqref{eq:killing time for free boundary exp(1)} for atomless initial conditions and $\Law_{u_0}(B_h\lvert \tau>h)$ is atomless for any $h>0$, 
\[
\frac{\Pm_{u_0}(\tau>h+t)}{\Pm_{u_0}(\tau>h)}=\Pm_{u_0}(\tau >h+t\lvert \tau>h)=e^{-t}
\]
for any $t,h>0$. Since $\tau>0$ almost surely, we can take the limit as $h\downarrow 0$ for fixed $t>0$ to see that
\[
\Pm_{u_0}(\tau> t)=\frac{\Pm_{u_0}(\tau>t)}{\Pm_{u_0}(\tau>0)}=\lim_{h\downarrow 0}\frac{\Pm_{u_0}(\tau>h+t)}{\Pm_{u_0}(\tau>h)}=\lim_{h\downarrow 0}e^{-t}=e^{-t},
\]
whence \eqref{eq:killing time for free boundary exp(1)} follows.\end{proof}

\end{appendix}

{\textbf{Acknowledgement:}} 
The authors are grateful to James Nolen for stimulating discussions on this problem. The work of OT was funded by the EPSRC MathRad programme grant EP/W026899/.

\bibliographystyle{alpha}
\bibliography{SelectionPrincipleBib.bib}
\end{document}